\documentclass[a4paper,10pt,fleqn]{article}

\usepackage[latin1]{inputenc}    
\usepackage[T1]{fontenc}      
\usepackage{geometry}         
\usepackage[english]{babel}  
\usepackage{amsfonts,amssymb,amsmath,amsthm}
\usepackage{geometry}
\usepackage{mathrsfs}
\usepackage{varioref}
\usepackage{dsfont}
\usepackage{graphicx}
\usepackage{color}

\newtheorem{théo}{Theorem}[section]
\newtheorem{prop}[théo]{Proposition}
\newtheorem{lemm}[théo]{Lemma}
\newtheorem{corr}[théo]{Corollary}
\newtheorem{défi}[théo]{Definition}

\theoremstyle{definition}
\newtheorem{rema}[théo]{Remark}
\newtheorem{hypo}{Hypothesis}[section]

\def\transpose{{}^t\!}
\DeclareMathOperator{\tr}{Tr}
\newcommand{\E}{\mathbb{E}}

\newcommand{\der}{\mathrm{d}}
\newcommand{\R}{\mathbb{R}}

\newcommand{\G}{\overline{G}}
\newcommand{\Pb}{\mathbb{P}}
\newcommand{\N}{\mathbb{N}}
\newcommand{\Q}{\mathbb{Q}}

\newcommand{\LL}{\mathscr{L}}

\newcommand{\F}{\mathscr{F}}

\numberwithin{equation}{section}

\begin{document}

\title{A probabilistic approach to large time behaviour of viscosity solutions of parabolic equations  with Neumann boundary conditions}

\author{Ying Hu\thanks{IRMAR, Université Rennes 1, Campus de Beaulieu, 35042 Rennes Cedex, France (ying.hu@univ-rennes1.fr), partially supported by Lebesgue center of mathematics ("Investissements d'avenir"
program - ANR-11-LABX-0020-01)} \and Pierre-Yves Madec\thanks{IRMAR, Université Rennes 1, Campus de Beaulieu, 35042 Rennes Cedex, France, partially supported by Lebesgue center of mathematics ("Investissements d'avenir"
program - ANR-11-LABX-0020-01)} }


\date{September 15, 2015}

\maketitle

\textsc{\textbf{Abstract:}}\ This paper is devoted to the study of the large time behaviour of viscosity solutions of parabolic equations  with Neumann boundary conditions. This work is the sequel of \cite{LARGE_TIME_BEHAVIOUR_HU_MADEC_RICHOU} in which a probabilistic method was developed to show that the solution of a parabolic semilinear PDE behaves like a linear term $\lambda T$ shifted with a function $v$, where $(v,\lambda)$ is the solution of the ergodic PDE associated to the parabolic PDE. We adapt this method in finite dimension by a penalization method in order to be able to apply an important basic coupling estimate result and with the help of a  regularization procedure in order to avoid the lack of regularity of the coefficients in finite dimension. The advantage of our method is that it gives an explicit rate of convergence.

\section{Introduction}
We are concerned with the large time behaviour of solutions of the Cauchy problem with Neumann boundary conditions:
\begin{align}\label{HJB Neumann cauchy changement de temps}
\left\{ 
\begin{array}{ll}
\frac{\partial u(t,x)}{\partial t} = \LL u(t,x) + f(x,\nabla u(t,x)\sigma) , & \forall (t,x) \in \R_+ \times G,\\
\frac{\partial u(t,x)}{\partial n} + g(x) = 0, &\forall (t,x) \in \R_+ \times \partial G, \\
u(0,x) = h(x),&\forall x \in \G,
\end{array}
\right.
\end{align}
where, at least formally, $\forall \psi : \G \rightarrow \R$,
\begin{align*}
(\LL \psi)(x) = \frac{1}{2}\tr (\sigma \transpose \sigma \nabla^2 \psi(x)) + \langle b(x), \nabla \psi(x) \rangle,
\end{align*}
and $G = \{ \phi >0 \}$ is a bounded convex open set of $\R^d$ with regular boundary. $u : \R_+ \times \G \rightarrow \R$ is the unknown function. 
We will assume that $b$ is Lipschitz and $\sigma$ is invertible. $h$ is continuous and  $g \in \mathscr{C}^1_{\text{lip}}(\G)$. Furthermore we will assume that the non-linear term $f(x,z) : \R^d \times \R^{1 \times d} \rightarrow \R$ is continuous in the first variable for all $z$ and there exists $C > 0$ such that for all $x \in \R^d$, $\forall z_1, z_2 \in \R^{1 \times d}$, $|f(x,z_1) - f(x,z_2)| \leq C|z_1 - z_2|$. Finally in order to obtain uniqueness for viscosity solutions of \eqref{HJB Neumann cauchy changement de temps}, we assume that $\partial G$ is $W^{3,\infty}$ and that there exists $ m \in \mathscr{C}((0,+\infty),\R)$, $m(0^+) = 0$ such that $\forall x,y \in \G, \forall z \in \R^{1 \times d}$,
\begin{align*}
|f(x,z) - f(y,z)| \leq m\left((1+|z|)|x-y|\right).
\end{align*}
A lot of papers deal with the large time behaviour of parabolic PDEs (see for e.g. \cite{NAMAH_ROQUEJOFFRE_CONVERGENCE}, \cite{FUJITA_LORETI_LONG_TIME_BEHAVIOR_QUADRATIC}, \cite{ISHII_ASYMPTOTIC_SOLUTIONS_LARGE_TIME_EUCLIDIAN}, \cite{FUJITA_ISHII_LORETI_ASYMPTOTIC},  \cite{ICHIHARA_SHEU_HAMILTON_JACOBI_BELLMAN} or \cite{CFH}), but there are not a lot of them which deal with Neumann boundary conditions. In \cite{BENACHOUR_DABULEANU_LARGE_TIME}, Benachour and Dabuleanu study the large time behaviour of the Cauchy problem with zero Neumann boundary condition
\begin{align}\label{EDP benachour dabuleanu}
\left\{ 
\begin{array}{ll}
\frac{\partial u(t,x)}{\partial t} = \Delta u(t,x) + a |\nabla u(t,x)|^p , & \forall (t,x) \in \R_+ \times G,\\
\frac{\partial u(t,x)}{\partial n}  = 0, &\forall (t,x) \in \R_+ \times \partial G, \\
u(0,x) = h(x),&\forall x \in \overline{G},
\end{array}
\right.
\end{align}
where $a \in \R$, $a \neq 0$, $p > 0$ and $G$ is a bounded open set with smooth boundary of $\mathscr{C}^3$ class. The large time behaviour depends on the exponent $p$. If $p \in (0,1)$, and if $h$ is a periodic function, then the solution is constant from a finite time. That is, there exist $T^* > 0$ and $c \in \R$ such that $u(t,x) = c$, for all $t > T^*$. When $p \geq 1$, any solution of \eqref{EDP benachour dabuleanu} converges uniformly to a constant, as $t \rightarrow + \infty$.

In \cite{ISHII_LONG_TIME_ASYMPTOTIC_SOLUTIONS_CONVEX_HJB}, Ishii establishes a result about the large time behaviour of a parabolic PDE in a bounded set with an Hamiltonian of first order $H(x,p)$, convex and coercive in $p$ and with Neumann boundary coniditons.

In \cite{BARLES_DA_LIO_BOUNDARY_ERGODIC_PROBLEM}, Barles and Da Lio give a result for the large time behaviour of \eqref{HJB Neumann cauchy changement de temps}. Moreover, the result about the large time behaviour has been  improved by Da Lio in \cite{DA_LIO_LARGE_TIME_NEUMANN} under the same hypotheses. In this last paper, the author studies the large time behaviour of non linear parabolic equation with Neumann boundary conditions on a smooth bounded domain $\G$:
\begin{align}\label{PDE DA LIO}
\left\{
\begin{array}{ll}
\frac{\partial u(t,x)}{\partial t} + F(x,\nabla u(t,x) , \nabla^2 u(t,x)) = \lambda, & \forall (t,x) \in \R_+ \times G,\\
L(x,\nabla u(t,x)) = \mu, & \forall (t,x) \in \R_+ \times \G,\\
u(0,x) = h(x), & \forall x \in G.
\end{array}
\right.
\end{align}
The spirit of this paper is slightly different from our work. Indeed, the result says that $\forall \lambda \in \R$, there exists $\mu \in \R$ such that \eqref{PDE DA LIO} has a continuous viscosity solution. Moreover there exists a unique $\widetilde{\lambda}$ such that $\mu(\widetilde{\lambda}) = \widetilde{\lambda}$ for which the solution of \eqref{PDE DA LIO} remains uniformly bounded in time $\widetilde{u}$. Then, there exists $\widetilde{u}_\infty$ solution of the ergodic PDE associated to \eqref{PDE DA LIO} such that 
\begin{align*}
\widetilde{u}(t,x) \underset{t \rightarrow + \infty}{ \longrightarrow} \widetilde{u}_{\infty}(x),~~~~\text{uniformly in } \G.
\end{align*}

We mention that no convergence rates are given in the above papers \cite{BENACHOUR_DABULEANU_LARGE_TIME,BARLES_DA_LIO_BOUNDARY_ERGODIC_PROBLEM,DA_LIO_LARGE_TIME_NEUMANN}.

Let us now state our main idea and result. Our method is purely probabilistic, which can be described as follows. First, let us consider $(X_t^x,K_t^x)_{t \geq 0}$ the solution of the following reflected SDE with values in $\overline{G} \times \R_+$,
\begin{align*}
\left\{ 
\begin{array}{ll}
 X_t^x = x + \int_0^t b(X_s^x) \der s +\int_0^t \nabla \phi(X_s^x) \der K_s^x + \int_0^t \sigma \der W_s, & t \geq 0, \\
K_t^x = \int_0^t \mathds{1}_{\left\{ X_s^x \in \partial G \right\}} \der K_s^x, ~~~~\forall t \geq 0,
\end{array}
\right.
\end{align*}
where $W$ is an $\R^d$-valued standard Brownian motion.
Let $(v,\lambda)$ be the solution of the following ergodic PDE,
\begin{align*}
\left\{
\begin{array}{l}
\LL v(x) + f(x,\nabla v(x) \sigma) -\lambda = 0, ~~~~\forall x \in G,\\
\frac{\partial v(t,x)}{\partial n} + g(x) = 0, ~~~~\forall x \in \partial G.
\end{array}
\right.
\end{align*}
Let $(Y^{T,x},Z^{T,x})$ be the solution of the BSDE:
\begin{align*}
\left\{
\begin{array}{l}
\der Y_s^{T,x} = -f(X_s^x,Z_s^{T,x}) \der s  - g(X_s^x) \der K_s^x + Z_s^{T,x} \der W_s,\\
Y_T^{T,x} = h(X_T^x),
\end{array}
\right.
\end{align*}
and $(Y^x,Z^x,\lambda)$ be solution of the EBSDE:
\begin{align*}
\der Y_s^x = -(f(X_s^x,Z_s^x)-\lambda) \der s - g(X_s^x) \der K_s^x  + Z_s^x \der W_s.
\end{align*}
Then we have the following probabilistic representation:
\begin{align*}
\left\{
\begin{array}{l}
Y_s^{T,x} = u(T-s,X_s^x),\\
Y_s^x = v(X_s^x).
\end{array}
\right.
\end{align*}
Then, in order to apply the method exposed in \cite{LARGE_TIME_BEHAVIOUR_HU_MADEC_RICHOU}, we penalize and regularize the reflected process in order to apply the basic coupling estimates. Then, the use of a stability argument for BSDE helps us to conclude. Finally, we deduce that there exists a constant $L \in \R$ such that for all $x \in \R^d$, 
\begin{align*}
Y_0^{T,x} - \lambda T - Y_0^x \underset{T \rightarrow + \infty}{\longrightarrow} L,
\end{align*}
i.e.
\begin{align*}
u(T,x) - \lambda T - v(x) \underset{T \rightarrow + \infty}{\longrightarrow} L.
\end{align*}
Our method also gives a rate of convergence:
\begin{align*}
|u(T,x) - \lambda T - v(x)| \leq Ce^{-\hat{\eta} T}.
\end{align*}

The main contributions of this paper are: (1) uniqueness of solution to the EBSDE by regularization of coefficients and by applying coupling estimate (see Theorem \ref{Uniqueness EBSDE} in Section 3);
(2) a probabilistic method to prove the uniqueness of solution to ergodic PDE (see Lemma \ref{Uniqueness of ergodic viscosity solutions} in Section 3); (3) an exponential rate of convergence for the large time 
behaviour of viscosity solutions of parabolic equations  with Neumann boundary conditions, which seems to be new comparing with convergence results in \cite{BENACHOUR_DABULEANU_LARGE_TIME,BARLES_DA_LIO_BOUNDARY_ERGODIC_PROBLEM,DA_LIO_LARGE_TIME_NEUMANN}.

\medskip

The paper is organized as follows: In section 2, we introduce some notations. In section 3, we recall some existence and uniqueness results about a perturbed SDE, a reflected SDE, a BSDE and an EBSDE that will be useful for what follow in the paper. We recall how such BSDE and EBSDE are linked with PDE. In section 4, we study the large time behaviour of the solution of the BSDE taken at initial time when the horizon $T$ of the BSDE increases. Then, we obtain a more precise result with an explicit rate of convergence in the Markovian case. In section 5, we apply our results to an optimal ergodic control problem.

\section{Notations}
We introduce some notations. Let $E$ be an Euclidian space. We denote by $\langle\cdot,\cdot \rangle$ its scalar product and by $|\cdot|$ the associated norm. We denote by $B(x,M)$ the ball of center $x \in E$ and radius $M>0$.  Given $\phi \in B_b(E)$, the space of bounded and measurable functions $\phi : E \rightarrow \R$, we denote by $||\phi ||_0 = \sup_{x \in E} |\phi(x)|$. If a function $f$ is continuous and defined on a compact and convex subset $\G$ of $\R^d$, we define $f_{\R^d} := f(\Pi(x))$ where $\Pi$ is the projection on $\G$. Note that $f_{\R^d}$ is continuous and bounded. $\mathscr{C}^k_{\text{lip}}(\G)$ denotes the set of the functions of class $\mathscr{C}^k$ whose partial derivatives of order $k$ are Lipschitz functions.

Given a probability space $(\Omega, \F, \Pb)$ with a filtration $\F_t$, we consider the following classes of stochastic processes.
\smallskip \\
1. $L_\mathscr{P}^p(\Omega, \mathscr{C}([0,T];E))$, $p \in [1,\infty)$, $T > 0$, is the space of predictable processes $Y$ with continuous paths on $[0,T]$ such that 
\begin{align*}
|Y|_{L_\mathscr{P}^p(\Omega, \mathscr{C}([0,T];E))} = \left(\E \sup_{t \in [0,T]} |Y_t|^p\right)^{1/p} < \infty.
\end{align*}
\smallskip \\
2. $L_\mathscr{P}^p(\Omega, {L}^2([0,T];E))$, $p \in [1,\infty)$, $T > 0$, is the space of predictable processes $Y$ on $[0,T]$ such that 
\begin{align*}
|Y|_{L_\mathscr{P}^p(\Omega, L^2([0,T],E))}  = \left\{\E \left( \int_0^T |Y_t|^2 \der t \right)^{p/2} \right\}^{1/p} < \infty.
\end{align*}
\smallskip \\
3. $L_{\mathscr{P},\text{loc}}^2(\Omega, {L}^2([0,\infty);E))$ is the space of predictable processes $Y$ on $[0,\infty)$ which belong to the space $L_\mathscr{P}^2(\Omega,L^2([0,T];E))$ for every $T >0$. We define in the same way $L_{\mathscr{P},\text{loc}}^p(\Omega, \mathscr{C}([0,\infty);E))$.

In the sequel, we consider a complete probability space $(\Omega, \F, \Pb)$ and a standard  Brownian motion denoted by $(W_t)_{t \geq 0}$  with values in $\R^d$. $(\F_t)_{t \geq 0}$ will denote the natural filtration of $W$ augmented with the family of $\Pb$-null sets of $\F$.

In this paper, $C$ denotes a generic constant for which we specify the dependency on some parameters when it is necessary to do so.
In this paper, we will consider only continuous viscosity solutions. 
\section{Preliminaries}
\subsection{The perturbed forward SDE}
Let us consider the following stochastic differential equation with values in  $\R^d$:
\begin{align}\label{SDE}
\left\{
\begin{array}{l}
\der X_t = d(X_t) \der t + b(t,X_t)\der t + \sigma \der W_t, ~~t \geq 0,\\
X_0 = x \in \R^d.
\end{array}
\right.
\end{align}
We will assume the following about the coefficients of the SDE:
\begin{hypo}\label{hypo SDE}
\begin{enumerate}
\item $d : \R^d \rightarrow \R^d$ is locally Lipschitz, strict dissipative (i.e. there exists $\eta >0$ such that for every $x,y \in \R^d$, $\langle d(x)-d(y),x-y \rangle \leq -\eta |x-y|^2$) and with polynomial growth (i.e. there exists $\mu>0$ such that for every $x \in \R^d$, $|d(x)| \leq C(1+|x|^{\mu})$).
\item $b : \R_+\times\R^d \rightarrow \R^d$ is bounded and measurable.
\item $\sigma \in \R^{d \times d}$ is invertible.
\end{enumerate}
\end{hypo}

\begin{défi}
We say that the SDE $(\ref{SDE})$ admits a weak solution if there exists a new $\F$-Brownian motion $(\widehat{W}^x)_{t \geq 0}$ with respect to a new probability measure $\widehat{\Pb}$ (absolutely continuous with respect to $\Pb$), and an $\F$-adapted process $(\widehat{X}^x)_{t \geq 0}$ with continuous trajectories for which $(\ref{SDE})$ holds with $(W_t)_{t\geq 0}$ replaced by $(\widehat{W}_t^x)_{t \geq 0}$.
\end{défi}
\begin{lemm}
Assume that Hypothesis \ref{hypo SDE} holds true and that $b(t,\cdot)$ is  Lipschitz uniformly w.r.t. $t \geq 0$. Then for every $x\in \R^d$, equation (\ref{SDE}) admits a unique strong solution, that is, an adapted $\R^d$-valued process denoted by $X^x$ with continuous paths satisfying $\Pb$-a.s.,
\begin{align*}
X_t^x = x + \int_0^t d(X_s^x) \der s + \int_0^t b(s,X_s^x) \der s + \int_0^t  \sigma \der W_s, ~~ \forall t \geq 0.
\end{align*}
Furthermore, we have the following estimate: $\forall s\ge 0$,
\begin{align}\label{estimee 1 sur X}
\E [|X_s^x|^p] \leq  C(1+|x|^p).
\end{align}
If $b$ is only bounded and measurable then there exists a weak solution $(\widehat{X},\widehat{W})$ and uniqueness in law holds. Furthermore, $(\ref{estimee 1 sur X})$
still holds (with respect to the new probability measure).
\end{lemm}
\begin{proof}
For the first part of the lemma see \cite{Hasminskii_STOCHASTIC_STABILITY}, Theorem 3.3 in Chapter 1 or \cite{MAO_STOCHASTIC}, Theorem 3.5.  Estimates (\ref{estimee 1 sur X})
is a simple consequence of Itô's formula.
Weak existence and uniqueness in law are a direct consequence of a Girsanov's transformation.
\end{proof}

We define the Kolmogorov semigroup associated to Eq. (\ref{SDE}) as follows: $\forall \phi : \R^d \rightarrow \R$ measurable with polynomial growth,
\begin{align*}
\mathscr{P}_t[\phi](x) = \E \phi(X_t^x).
\end{align*}

\begin{lemm}[Basic coupling estimate]\label{basic coupling estimates}
Assume that Hypothesis \ref{hypo SDE} holds true and that $b(t,\cdot)$ is Lipschitz uniformly w.r.t. $t \geq 0$. Then there exists $\hat{c}> 0$ and $\hat{\eta}>0$ such that for all $\phi : \R^d \rightarrow \R$ measurable and bounded,
\begin{align}\label{equation basic coupling estimate}
|\mathscr{P}_t[\phi](x) - \mathscr{P}_t[\phi](y)| \leq \hat{c}(1+|x|^2+|y|^2)e^{- \hat{\eta} t}\sup_x|\phi(x)|.
\end{align}
We stress the fact that $\hat{c}$ and $\hat{\eta}$ depend on $b$ only through $\sup_{t \geq 0} \sup_{x \in \R^d}|b(t,x)|$.
\end{lemm}
\begin{proof}
See \cite{MADEC_ERGODIC}.
\end{proof}

\begin{corr}\label{basic coupling estimates étendu}
Relation (\ref{equation basic coupling estimate}) can be extended to the case in which $b$ is only bounded and measurable and for all $t \geq 0$, there exists a uniformly bounded sequence of Lipschitz functions in $x$, $(b_n(t,\cdot))_{n \geq 1}$ (i.e. $\forall n \in \N$, $b_n(t,\cdot)$ is Lipschitz uniformly w.r.t. $t \geq 0$ and $\sup_n\sup_t \sup_x |b_n(t,x)| < + \infty$) such that
\begin{align*}
\lim_n b_n(t,x) = b(t,x),~~~~~ \forall t \geq 0, \forall x \in R^d.
\end{align*}
Clearly in this case in the definition of $\mathscr{P}_t[\phi]$ the mean value is taken with respect to the new probability measure $\widehat{\Pb}$.
\end{corr}
\begin{proof}
It is enough to adapt the proof of Corollary $2.5$ in \cite{DEBUSSCHE_HU_TESSITORE_ERGODIC_BSDE_WEAKLY_DISSIPATIVE}. The goal is to show that, if $\mathscr{P}^n$ denotes the Kolmogorov semigroup corresponding to equation (\ref{SDE}) but with $b$ replaced by $b_n$, then $\forall x \in \R^d$ , $\forall t \geq 0$,
\begin{align*}
\mathscr{P}_t^n[\phi](x) \underset{n \rightarrow + \infty}{ \longrightarrow} \mathscr{P}_t[\phi](x).
\end{align*}
\end{proof}

\begin{rema}\label{basic coupling estimate etendue double suite}
Similarly, if there exists a uniformly bounded sequence of Lipschitz functions $(b_{m,n}(t,\cdot))_{m \in \N, n \in \N}$ (i.e. $\forall n \in \N, \forall m \in \N$, $b_{m,n}(t, \cdot)$ is Lipschitz uniformly w.r.t. $t \geq 0$  and $\sup_m \sup_n \sup_t \sup_x |b_{m,n}(t,x)| < + \infty$) such that
\begin{align*}
\lim_m \lim_n b_{m,n}(t,x) = b(t,x), ~~~~~\forall t \geq 0, \forall x \in \R^d,
\end{align*} 
then, if $\mathscr{P}^{m,n}$ is the Kolmogorov semigroup corresponding to equation $(\ref{SDE})$ but with $b$ replaced by $b_{m,n}$, we have $\forall t \geq 0$, $\forall x \in \R^d$,
\begin{align*}
\lim_m \lim_n \mathscr{P}_t^{m,n}[\phi](x) = \mathscr{P}[\phi](x),
\end{align*}
which shows that relation  (\ref{equation basic coupling estimate}) still holds.
\end{rema}

We will need to apply the lemma above to some functions with particular form.
\begin{lemm}\label{approximation upsilon}
Let $f : \R^d \times \R^{1 \times d} \rightarrow \R$ be continuous in the first variable and Lipschitz in the second one and $\zeta, \zeta'$ be two continuous functions: $\R_+ \times \R^d \rightarrow \R^{1 \times d}$ be such that for all $s \geq 0$, $\zeta(s,\cdot)$ and $\zeta'(s,\cdot)$ are continuous. We define, for every $s \geq 0$ and $x \in \R^d$,
\begin{align*}
\Upsilon(s,x) = \left\{ 
\begin{array}{ll}
\frac{f(x,\zeta(s,x))-f(x,\zeta'(s,x))}{|\zeta(s,x)-\zeta'(s,x)|^2}\transpose (\zeta(s,x)-\zeta'(s,x)), &\text{if } \zeta(s,x) \neq \zeta'(s,x),\\
0, & \text{if } \zeta(s,x) = \zeta'(s,x).
\end{array}
\right.
\end{align*}
Then, there exists a uniformly bounded sequence of Lipschitz functions  $(\Upsilon_{m,n}(s,\cdot))_{m \in \N, n \in \N}$ (i.e., for every $m \in \N^*$ and $n \in \N^*$, $\Upsilon_{m,n}(s,\cdot)$ is Lipschitz and $\sup_m \sup_n \sup_s \sup_x |\Upsilon_{m,n}(s,x)| < + \infty$) such that for every $s \geq 0$ and for every $x \in \R^d$, 
\begin{align*}
\forall x \in \R^d, \lim_{m \rightarrow + \infty}\lim_{n \rightarrow + \infty} \Upsilon_{m,n}(s,x) = \Upsilon(s,x).
\end{align*}
\end{lemm}
\begin{proof}
See the proof of Lemma $3.5$ in \cite{DEBUSSCHE_HU_TESSITORE_ERGODIC_BSDE_WEAKLY_DISSIPATIVE}.
\end{proof}

\subsection{The reflected SDE}
We consider a process $X_t^x$ reflected in $\G = \left\{ \phi >0 \right\}$. Let $(X_t^x,K_t^x)_{t \geq 0}$ denote the unknown of the following SDE:
\begin{align}\label{SDE_reflected}
\left\{ 
\begin{array}{ll}
 X_t^x = x + \int_0^t b(X_s^x) \der s +\int_0^t \nabla \phi(X_s^x) \der K_s^x + \int_0^t \sigma \der W_t, & t \in \R_+, \\
K_t^x = \int_0^t \mathds{1}_{\left\{ X_s^x \in \partial G \right\}} \der K_s^x.
\end{array}
\right.
\end{align}

\begin{hypo}\label{hypo SDE reflected}
\begin{enumerate}
\item $b : \G \rightarrow \R^d$ is Lipschitz.
\item $\sigma \in \R^{d \times d}$ is invertible.
\end{enumerate}
\end{hypo}

We will make the following assumptions about $G$.
\begin{hypo}\label{hypo_G}
\begin{enumerate}
\item $G$ is a bounded convex open set of $\R^d$.
\item $\phi \in \mathscr{C}^2_{\text{lip}}(\R^d)$ and $G = \left\{ \phi > 0 \right\}$, $\partial G = \left\{ \phi = 0 \right\}$ and $\forall x \in \partial G$, $|\nabla \phi(x)| = 1$.
\end{enumerate}
\end{hypo}

\begin{rema}\label{remarque transformation b en dissipatif}
Let us denote by $\Pi(x)$ the projection of $x \in \R^d$ on $\G$. Let us extend the definition of $b$ to $\R^d$ by setting, $\forall x \in \R^d$,
\begin{align*}
\widetilde{b}(x) &:= -x + (b(\Pi(x)) + \Pi(x)).
\end{align*}
Note that $d(x) :=  -x$ is strictly dissipative and that $p(x) := b(\Pi(x)) + \Pi(x)$ is Lipschitz and bounded. Therefore, $\widetilde{b}$ is weakly dissipative (a function is called weakly dissipative if it is a sum of a strictly dissipative function and a bounded one),
and satisfies Hypothesis \ref{hypo SDE}.
\end{rema}

Let us denote by $(X_t^{x,n})$ the solution of the following penalized SDE associated with \eqref{SDE_reflected}:
\begin{align*}
X_t^{x,n} = x + \int_0^t  \left[\widetilde{b}(X_s^{x,n}) +F_n(X_s^{x,n})\right]\der s + \int_0^t \sigma\der W_s,
\end{align*}
where $\forall x \in \R^d$, $F_n(x) = -2 n (x - \Pi(x))$.

\begin{lemm}
Assume that the Hypotheses \ref{hypo SDE reflected} and \ref{hypo_G} hold true. Then for every $x \in \G$ there exists a unique pair of processes $(X_t^x,K_t^x)_{t \geq 0}$ with values in $(\G \times \R_+)$ and which belongs to the space $L_{\mathscr{P},\text{loc}}^p(\Omega, \mathscr{C}([0,+\infty[;\R^d)) \times L_{\mathscr{P},\text{loc}}^p(\Omega, \mathscr{C}([0,+\infty[;\R_+))$ , $\forall p \in [1,+\infty [$, satisfying $(\ref{SDE_reflected})$ and such that
\begin{align*}
\eta_t^x := \int_0^t \nabla \phi(X_s^x) \der K_s^x, ~~~~\text{ has bounded variation on $[0,T]$},~\forall 0 \leq T < + \infty, ~~\eta_0^x = 0,
\end{align*}
and for all process $z$ continuous and progressively measurable taking values in the closure $\G$ we have
\begin{align*}
\int_0^T (X_s^x - z_s) \der K_s^x \leq 0,~~~~ \forall T \geq 0. 
\end{align*}
Finally, the following estimates holds for the convergence of the penalized process: for any $1 < q < p/2$, for any $T \geq 0$ there exists $C \geq 0$ such that
\begin{align*}
\E \sup_{0 \leq t \leq T} |X_t^{x,n} - X_t^x|^p \leq \frac{C}{n^q}.
\end{align*}
\end{lemm}
\begin{proof}
See Lemma 4.2 in \cite{MADEC_ERGODIC}.  
\end{proof}

\subsection{The BSDE}
Let us fix $T>0$ and let us consider the following BSDE in finite horizon for an unknown process $(Y_s^{T,t,x},Z_s^{T,t,x})_{s \in [t,T]}$ with values in $\R \times \R^{1 \times d}$:
\begin{align}\label{BSDE}
Y_s^{T,t,x} = \xi^T + \int_s^T f(X_r^{t,x},Z_r^{T,t,x}) \der r + \int_s^T g(X_r^{t,x}) \der K_r^{t,x} - \int_s^T Z_r^{T,t,x} \der W_r,~~~~ \forall s \in [t,T],
\end{align}
where $(X_s^{t,x},K_s^{t,x})_{s \in [t,T]}$ is the solution of the SDE (\ref{SDE_reflected}) starting from $x$ at time $t$. If $t=0$, we use the following standard notations $X_s^{x} = X_s^{0,x}$,$K_s^{x} = K_s^{0,x}$, $Y_s^{T,x} := Y_s^{T,0,x}$ and $Z_s^{T,x} = Z_s^{T,0,x}$.
We will assume the following assumptions.
\begin{hypo}[Path dependent case]\label{hypo_path_dependent}
There exists $C > 0$, such that the function $f : \G \times \R^{1 \times d} \rightarrow \R$ and $\xi^T$ satisfy:
\begin{enumerate}
\item $\xi^T$ is a real-valued random variable $\mathscr{F}_T$ measurable and $|\xi^T| \leq C$.
\item $\forall x \in \G$, $\forall z_1, z_2 \in \R^{1 \times d}$, $|f(x,z_1) - f(x,z_2)| \leq C|z_1 - z_2|$.
\item $\forall z \in \R^{1 \times d}$, $f(\cdot,z)$ is continuous.
\item $g \in \mathscr{C}^1_{\text{lip}}(\G)$.
\end{enumerate}
\end{hypo}

\begin{lemm}
Assume that the Hypotheses \ref{hypo SDE reflected}, \ref{hypo_G} and \ref{hypo_path_dependent} hold true, then there exists a unique solution $(Y_s^{T,t,x},Z_s^{T,t,x}) \in$  $L_\mathscr{P}^2(\Omega, \mathscr{C}([0,T];\R)) \times L_\mathscr{P}^2(\Omega, {L}^2([0,T];\R^{1\times d}))$.
\end{lemm}
\begin{proof}
See Theorem 1.7 in \cite{PARDOUX_ZHANG_GENERALIZED}.
\end{proof}

\begin{hypo}[Markovian case]\label{hypo BSDE markovien}
There exists $C >0$ such that
\begin{enumerate}
\item $\xi^T = h(X_T^x)$, where $h : \G \rightarrow \R $ is continuous.
\item $\forall x \in \G$, $\forall z,z' \in \R^{1 \times d}$, $|f(x,z)-f(x,z')|\leq C|z-z'|$.
\item $\forall z \in \R^{1 \times d}$, $f(\cdot,z)$ is continuous.
\item $g \in \mathscr{C}^1_{\text{lip}}(\G)$.
\end{enumerate}
\end{hypo}

Let us consider the following semilinear PDE:
\begin{align}\label{PDE finite horizon}
\left\{
\begin{array}{ll}
\frac{\partial u(t,x)}{\partial t} + \LL u(t,x) + f(x,\nabla u(t,x) \sigma) = 0, & \forall (t,x) \in [0,T]\times G,\\
\frac{\partial u(t,x)}{\partial n} + g(x) = 0, &\forall (t,x) \in [0,T] \times \partial G, \\
u(T,x) = h(x), &\forall x \in G,
\end{array}
\right.
\end{align}
where $\LL u(t,x) = \frac{1}{2} \tr(\sigma \transpose \sigma \nabla^2 u(t,x)) +  \langle b(x) , \nabla u(t,x) \rangle$.

\begin{lemm}[Existence]\label{lemme exitence sol de viscosité parabolique}
Assume that the Hypotheses \ref{hypo SDE reflected}, \ref{hypo_G} and \ref{hypo BSDE markovien} hold true, then there exists a continuous viscosity solution to the PDE (\ref{PDE finite horizon}) given by
\begin{align*}
u_T(t,x) = Y_t^{T,t,x}.
\end{align*}
\end{lemm}
\begin{proof}
In our framework, $u_T(t,x) \in \mathscr{C}([0,T]\times \G ; \R)$. Indeed, first as in the proof of Theorem 3.1 in \cite{Richou_Ergodic_BSDE_PDe_Neumann}, we deduce the existence of a function $v^1 : \G \rightarrow \R$ which belongs to the space $\mathscr{C}^2_{\text{lip}}(\G)$ and which is solution of Helmholtz's equation for some $\alpha \in \R$,
\begin{align*}
\left\{
\begin{array}{ll}
\Delta v^1(x) - \alpha v^1(x) =  0,\\
\frac{\partial v^1(x)}{\partial n} + g(x) = 0.
\end{array}
\right.
\end{align*}
We set $Y_s^{1,t,x} = v^1(X_s^{t,x})$ and $Z_s^{1,t,x} = \nabla v^1(X_s^{t,x}) \sigma$. These processes verify, $\forall s \in [t,T]$,
\begin{align*}
Y_s^{1,t,x} = v^1(X_T^{t,x}) + \int_s^T [-\mathscr{L} v^1(X_r^{t,x})] \der r + \int_s^T g(X_r^{t,x}) \der K_r^{t,x} - \int_s^T Z_r^{1,t,x} \der W_r,
\end{align*}
where 
\begin{align*}
\LL v^1(x) = \frac{1}{2}\tr (\sigma \transpose \sigma \nabla^2 v^1(x)) + \langle b(x), \nabla v^1(x) \rangle.
\end{align*}
Then, if we define
\begin{align*}
&\widetilde{Y}_s^{T,t,x} = Y_s^{T,t,x} - v^1(X_s^{t,x}),\\
&\widetilde{Z}_s^{T,t,x} = Z_s^{T,t,x} - \nabla v^1(X_s^{t,x}) \sigma,
\end{align*}
$(\widetilde{Y}^{T,t,x},\widetilde{Z}^{T,t,x})$ satisfies the BSDE, $\forall s \in [t,T]$:
\begin{align*}
\widetilde{Y}_s^{T,t,x} = (h-v^1)(X_T^{t,x}) &+ \int_s^T \left[f(X_r^{t,x},\widetilde{Z}_r^{T,t,x} + \nabla v^1(X_r^{t,x}) \sigma) + \mathscr{L}v^1(X_r^{t,x}) \right] \der r\\
&- \int_s^T \widetilde{Z}_r^{T,t,x} \der W_r,
\end{align*}
which shows, since $v^1 \in \mathscr{C}^2_{\text{lip}}(\G)$, that $\left((t,x) \mapsto \widetilde{Y}_t^{T,t,x}\right)$ is continuous. To show that $u_T(t,x)$ is a viscosity solution of \eqref{PDE finite horizon} see \cite{PARDOUX_ZHANG_GENERALIZED}, Theorem 4.3.
\end{proof}
Uniqueness for solutions of \eqref{PDE finite horizon} holds under additional assumptions in our framework.
\begin{hypo}\label{hypo unicite parabolic PDE}
\begin{enumerate}
\item $\partial G$ is of class $W^{3,\infty}$.
\item $\exists m \in \mathscr{C}((0,+\infty),\R)$, $m(0^+) = 0$ such that $\forall x,y \in \G, \forall z \in \R^{1 \times d}$,
\begin{align*}
|f(x,z) - f(y,z)| \leq m\left((1+|z|)|x-y|\right).
\end{align*}
\end{enumerate}
\end{hypo}

\begin{lemm}[Uniqueness]\label{lemme unicite sol de viscosité parabolique}
Assume that the Hypotheses \ref{hypo SDE reflected}, \ref{hypo_G}, \ref{hypo BSDE markovien} and \ref{hypo unicite parabolic PDE} hold true. Then, uniqueness holds for viscosity solutions of \eqref{PDE finite horizon}.
\end{lemm}
\begin{proof}
See Theorem II.1 in \cite{BARLES_FULLY_NONLINEAR_NEUMANN}.
\end{proof}

\begin{rema}
By the following change of time: $\widetilde{u}_T(t,x) := u_T(T-t,x)$, we remark that $\widetilde{u}_T(t,x)$ is the unique viscosity solution of (\ref{HJB Neumann cauchy changement de temps}). Now remark that $\widetilde{u}_T(T,x) = u_T(0,x) = Y_0^{T,0,x} = Y_0^{T,x}$, therefore the large time behaviour of $Y_0^{T,x}$ is the same as that of the solution of equation (\ref{HJB Neumann cauchy changement de temps}).
\end{rema}

\subsection{The EBSDE}
In this section, we consider the following ergodic BSDE for an unknown process $(Y_t^x,Z_t^x,\lambda)_{t \geq 0}$ with values in $\R \times \R^{1\times d}\times \R$:
\begin{align}\label{EBSDE}
Y_t^x=Y_T^x+\int_t^T(f(X_s^x,Z_s^x) -\lambda)\der s + \int_t^T g(X_s^x) \der K_s^x - \int_t^T Z_s^x \der W_s,~~~~\forall 0 \leq t \leq T < + \infty.
\end{align}

\begin{hypo}\label{hypo EBSDE} There exists $C>0$ such that,
\begin{enumerate}
\item $\forall x \in \G$, $\forall z,z' \in \R^{1 \times d}$, $|f(x,z)-f(x,z')|\leq C|z-z'|$.
\item $\forall z \in \R^{1 \times d}$, $f(\cdot,z)$ is continuous.
\item $g \in \mathscr{C}^1_{\text{lip}}(\G)$.
\end{enumerate}
\end{hypo}
Without loss of generality, we assume that $0 \in \G$.
\begin{lemm}[Existence when Neumann boundary conditions are null]\label{lemm_existence_EBSDE_Neumann_null}
Assume that $g \equiv 0$ and that the Hypotheses \ref{hypo SDE reflected}, \ref{hypo_G} and \ref{hypo EBSDE} hold true. Then there exists a solution $(Y^{x},Z^x,\lambda) \in L_{\mathscr{P},\text{loc}}^2(\Omega, \mathscr{C}([0,+\infty[;\R)) \times L_{\mathscr{P},\text{loc}}^2(\Omega, {L}^2([0,+\infty[;\R^{1\times d})) \times \R$ to (\ref{EBSDE}). Moreover there exist $v : \G \rightarrow \R$ and $\xi : \G \rightarrow \R^{1 \times d}$ measurable  such that for every $x,y \in \G$, for all $t \geq 0$,
 \begin{align*}
&Y_t^{x} = v(X_t^x), Z_t^x = \xi(X_t^x),\\
&v(0) = 0,\\
&|v(x) - v(y)| \leq C,\\
&|v(x) - v(y)| \leq C|x-y|.
\end{align*}
\end{lemm}
\begin{proof}
First let us recall that by Remark \ref{remarque transformation b en dissipatif}, one can replace $b$ by its extension $\widetilde{b}$ which is weakly dissipative. Therefore, replacing $f$ by $f_{\R^d}$, we obtain, by Theorem $4.4$ in \cite{MADEC_ERGODIC} that there exists $v : \G \rightarrow \R$ and $\xi : \G \rightarrow \R^{1 \times d}$ measurable  such that for every $x,y \in \G$, for all $t \geq 0$,
 \begin{align*}
&Y_t^{x} = v(X_t^x), Z_t^x = \xi(X_t^x),\\
&v(0) = 0,\\
&|v(x) - v(y)| \leq C(1+|x|^{2} + |y|^{2}),\\
&|v(x) - v(y)| \leq C(1+|x|^{2} + |y|^{2})|x-y|.
\end{align*}
And the result follows by the boundedness of $\G$.
\end{proof}

\begin{lemm}[Existence]\label{lemm_existence_EBSDE}
Assume that the Hypotheses \ref{hypo SDE reflected}, \ref{hypo_G} and \ref{EBSDE} hold true. Then there exists a solution $(Y^{x},Z^x,\lambda) \in L_{\mathscr{P},\text{loc}}^2(\Omega, \mathscr{C}([0,+\infty[;\R)) \times L_{\mathscr{P},\text{loc}}^2(\Omega, {L}^2([0,+\infty[;\R^{1\times d})) \times \R$ to the EBSDE (\ref{EBSDE}). Moreover there exists $v : \G \rightarrow \R$ such that for every $x,y \in \G$, for all $t \geq 0$,
\begin{align*}
&Y_t^{x} = v(X_t^x),\\
&|v(x)| \leq C,\\
&|v(x) - v(y)| \leq C|x-y|.
\end{align*}
\end{lemm}
\begin{proof}
First as in the proof of Theorem 3.1 in \cite{Richou_Ergodic_BSDE_PDe_Neumann}, we deduce the existence of a function $v^1 : \G \rightarrow \R$ which belongs to the space $\mathscr{C}^2_{\text{lip}}(\G)$ and is solution of Helmholtz's equation for some $\alpha \in \R$,
\begin{align*}
\left\{
\begin{array}{ll}
\Delta v^1(x) - \alpha v^1(x) =  0,\\
\frac{\partial v^1(x)}{\partial n} + g(x) = 0.
\end{array}
\right.
\end{align*}
Then, if we define $(Y^1_t := v^{1}(X_t^x),Z_t^1 := \nabla v^1(X_t^x) \sigma)$, $(Y^1,Z^1)$ satisfies, for every $0 \leq t \leq T < + \infty$:
\begin{align}\label{EBSDE 1}
Y^1_t = Y_T^1 + \int_t^T \left[-\LL v^1(X_s^x) \right]\der s + \int_t^T g(X_s^x) \der K_s^x - \int_t^T Z^1_s \der W_s, 
\end{align}
where
\begin{align*}
(\LL v^1)(x) = \frac{1}{2}\tr (\sigma \transpose \sigma \nabla^2 v^1) + \langle \widetilde{b}(x), \nabla v^1 \rangle.
\end{align*}
Now consider the following EBSDE:
\begin{align}\label{EBSDE 2}
Y^2_t = Y_T^2 + \int_t^T [f^2(X_s^x,Z_s^2) -\lambda ] \der s - \int_t^T Z_s^2 \der W_s, ~~~~\forall 0 \leq t \leq T < + \infty,
\end{align}
with $f^2(x,z) := \LL v^1(x) + f(x, z + \nabla v^1(x) \sigma)$. Since $\forall z \in \R^{1 \times d}$, $f^2(\cdot,z)$ is continuous and since for every $x \in \G$, $f^2(x,\cdot)$ is Lipschitz, one can apply Lemma \ref{lemm_existence_EBSDE_Neumann_null} to obtain the existence of a solution $(Y^2_t = v^2(X_t^x), Z^2_t = \xi^2(X_t^x) )$ to EBSDE (\ref{EBSDE 2}) such that $v^2$ is continuous. We set
\begin{align*}
&Y_t^x = Y^1_t + Y_t^2 = v^1(X_t^x) + v^2(X_t^x),\\
&Z_t^x = Z_t^1 + Z_t^2 = \nabla v^1(X_t^x) \sigma + \xi^2(X_t^x).
\end{align*}
Then $(Y^x,Z^x,\lambda)$ is a solution of the EBSDE (\ref{EBSDE}).
\end{proof}

\begin{théo}[Uniqueness of $\lambda$]\label{theorem uniqueness for lambda EBSDE}Assume that the Hypotheses \ref{hypo SDE reflected}, \ref{hypo_G} and \ref{EBSDE} hold true. If $(Y^1,Z^1,\lambda^1)$ and $(Y^2,Z^2,\lambda^2)$ denote two solutions of the EBSDE \eqref{EBSDE} in the class of solutions $(Y,Z,\lambda)$ such that $\forall t \geq 0, |Y_t| \leq C$, $\Pb$-a.s. and $Z \in L^2_{\mathscr{P},loc}(\Omega, L^2([0,\infty[;\R^{1 \times d})$, then
\begin{align*}
\lambda^1 = \lambda^2.
\end{align*}
\end{théo}
\begin{proof}
See Theorem $4.6$ in \cite{FUHRMAN_HU_TESSITORE_ERGODIC_BSDE}.
\end{proof}

Let us now state our main result of this section.
\begin{théo}[Uniqueness of solutions $(Y,Z,\lambda)$]\label{Uniqueness EBSDE}Assume that the Hypotheses \ref{hypo SDE reflected}, \ref{hypo_G} and \ref{EBSDE} hold true. Uniqueness holds for solutions $(Y,Z,\lambda)$ of the EBSDE \eqref{EBSDE} in the class of solutions such that there exists $v : \G \rightarrow \R$ continuous, $Y_s = v(X_s^x)$ with $v(0) = 0$, and $Z \in L^2_{\mathscr{P},loc}(\Omega, L^2([0,\infty[;\R^{1 \times d})$.
\end{théo}
\begin{proof}
Let $(Y^1 = v^1(X^x), Z^1 , \lambda^1)$ and $(Y^2 = v^2(X^x), Z^2 , \lambda^2)$ denote two solutions. Then from Theorem \ref{theorem uniqueness for lambda EBSDE}, we deduce that $\lambda^1 = \lambda^2 =: \lambda$.

Now, let us denote by $v : \G \rightarrow \R$, $ v \in \mathscr{C}^2_{\text{lip}}(\G)$ and solution of Helmholtz's equation for some $\alpha \in \R$
\begin{align*}
\left\{
\begin{array}{ll}
\Delta v(x) - \alpha v(x) =  0,\\
\frac{\partial v(x)}{\partial n} + g(x) = 0.
\end{array}
\right.
\end{align*}
Then, if we define $(Y_t := v(X_t^x),Z_t := \nabla v(X_t^x) \sigma)$, $(Y,Z)$ satisfies, for every $0 \leq t \leq T < + \infty$:
\begin{align}\label{EBSDE Helmholtz 1}
Y_t = Y_T + \int_t^T \left[-\LL v(X_s^x) \right]\der s + \int_t^T g(X_s^x) \der K_s^x - \int_t^T Z_s \der W_s, 
\end{align}
where
\begin{align*}
(\LL v)(x) = \frac{1}{2}\tr (\sigma \transpose \sigma \nabla^2 v) + \langle \widetilde{b}(x), \nabla v \rangle.
\end{align*}
Therefore, $(\widehat{Y}_t^1 = Y_t^1 - v(X_t^x) , \widehat{Z}^1_t = Z_t^1 - \transpose \nabla v(X_t^x) \sigma)$ satisfies the BSDE, $\forall 0 \leq t \leq T < + \infty$,
\begin{align*}
\widehat{Y}_t^1 = \widehat{Y}_T^1 + \int_t^T \widehat{f}(X_s^x,\widehat{Z}_s^x)  \der s - \int_t^T \widehat{Z}_s^1 \der W_s,
\end{align*}
where $\forall x,z \in \R^d \times \R^{1 \times d}$,
\begin{align*}
\widehat{f}(x,z) = f\left(x,z + \transpose \nabla v^1(x) \sigma\right) - \lambda + \LL v(x).
\end{align*}
Then, let $(\widehat{Y}^{1,T,t,x}, \widehat{Z}^{1,T,t,x})$ be the solution of the following BSDE, $\forall s \in [t,T]$,
\begin{align*}
\widehat{Y}_s^{1,T,t,x} = (v^1 - v)(X_T^{t,x}) + \int_t^T \widehat{f}(X_s^x,\widehat{Z}_s^{1,T,t,x}) \der s - \int_t^T \widehat{Z}_s^{1,T,t,x} \der W_s.
\end{align*}
By uniqueness of solutions to BSDE, we deduce that
\begin{align*}
v^{1}(x) - v(x) = \widehat{Y}_0^{1,T,0,x}.
\end{align*}
Now, we fix infinitely differentiable functions $\rho_{\varepsilon} : \R^d \rightarrow \R^+$ bounded together with their derivatives of all order, such that: $\int_{\R^d}\rho_{\varepsilon}(x) \der x = 1$ and
\begin{align*}
\text{supp}(\rho_{\varepsilon}) \subset \left\{ \xi \in \R^d: |\xi| \leq \varepsilon \right\},
\end{align*}
where $\text{supp}$ denotes the support. Then we define $\forall n \in \N$,
\begin{align*}
&(F_n)_{\varepsilon} (x) = \int_{\R^d} \rho_{\varepsilon}(y) F_n(x-y) \der y,\\
&\widetilde{b}_{\varepsilon}(x) = \int_{\R^d} \rho_{\varepsilon}(y) \widetilde{b}(x-y) \der y.
\end{align*}
Let us denote by $X^{t,x,n,\varepsilon}$ the solution of the following SDE, $\forall s \geq t$,
\begin{align*}
X_s^{t,x,n,\varepsilon} = x + \int_t^s \left( \widetilde{b}_\varepsilon + (F_n)_\varepsilon  \right)(X_r^{t,x,n,\varepsilon}) \der r + \int_t^s \sigma \der W_r,
\end{align*}
and let $(Y^{1,T,t,x,n,\varepsilon}, Z^{1,T,t,x,n,\varepsilon})$ be the solution of the following BSDE, $\forall s \in [t,T]$,
\begin{align*}
Y^{1,T,t,x,n,\varepsilon}_s = (v^1-v)(X_s^{t,x,n,\varepsilon}) + \int_t^s \widehat{f}(X_r^{t,x,n,\varepsilon},Z_r^{1,T,t,x,n,\varepsilon}) \der r - \int_s^T Z_r^{1,T,t,x,n,\varepsilon} \der W_s.
\end{align*}
Then by a stability result, (see for e.g. Lemma 2.3 of \cite{BRIAND_HU_STABILITY_BSDE}), we deduce that
\begin{align}\label{convergence de Y1 pour l unicite EDSRE}
\lim_{n\rightarrow + \infty}\lim_{\varepsilon \rightarrow + \infty} Y^{1,T,0,x,n,\varepsilon}_0 = \widehat{Y}_0^{1,T,0,x} = v^{1}(x) - v(x). 
\end{align}

Similarly, defining $(Y^{2,T,t,x}, Z^{2,T,t,x})$ and $(Y^{2,T,t,x,n,\varepsilon}, Z^{2,T,t,x,n,\varepsilon})$ in the same way, we deduce that
\begin{align*}
\lim_{n\rightarrow + \infty}\lim_{\varepsilon \rightarrow + \infty} Y^{2,T,0,x,n,\varepsilon}_0 = \widehat{Y}_0^{2,T,0,x} = v^{2}(x) - v(x). 
\end{align*}

Furthermore, by Theorem 4.2 (or Theorem 4.2 in \cite{FUHRMAN_TESSITORE_BISMUT_ELWORTHY}), if we define $u^{1,T, n,\varepsilon}(t,x) := Y^{1,T,t,x,n,\varepsilon}_t$, then $(x \mapsto u^{1,T, n,\varepsilon}(t,x))$ is continuously differentiable for all $t \in [0,T[$, and  $\forall s \in [t,T[$,
\begin{align*}
Z_s^{1,T,t,x,n,\varepsilon} = \transpose \nabla u^{1,T, n,\varepsilon}(s,X_s^{t, x , n,\varepsilon }) \sigma .
\end{align*} 
Similarly, we define $u^{2,T, n,\varepsilon}(t,x) := Y^{2,T,t,x,n,\varepsilon}_t$ and then
\begin{align*}
Z_s^{2,T,t,x,n,\varepsilon} = \transpose \nabla u^{2,T, n,\varepsilon}(s,X_s^{t, x , n,\varepsilon }) \sigma .
\end{align*}
Therefore, taking $t= 0$, $\forall T > 0$,
\begin{align*}
 u^{1,T, n,\varepsilon}(0,x) -  u^{2,T, n,\varepsilon}(0,x) &= (v^1-v^2)(X_T^{x,n,\varepsilon}) - \int_0^T( Z_s^{1,T,x,n,\varepsilon} - Z_s^{2,T,x,n,\varepsilon})  \der W_s \\
 &~~~~+ \int_0^T \left[\widehat{f}(X_s^{x,n,\varepsilon},Z_s^{1,T,x,n,\varepsilon}) - \widehat{f}(X_s^{x,n,\varepsilon},Z_s^{2,T,x,n,\varepsilon})\right] \der s\\
 &=(v^1-v^2)(X_T^{x,n,\varepsilon}) \\
 &~~~~ - \int_0^T (Z_s^{1,T,x,n,\varepsilon} - Z_s^{2,T,x,n,\varepsilon}) (-\beta(s,X_s^{x,n,\varepsilon}) \der s + \der W_s),
\end{align*}
where 
\begin{align*}
\beta^T(s,x) = 
\left\{
\begin{array}{lll}
&\!\!\!\!\!\!\!\!\!\!\!\!\!\frac{(f(x,\transpose\nabla u^{1,T, n,\varepsilon}(s,x)\sigma)  - f(x,\transpose\nabla u^{2,T, n,\varepsilon}(s,x)\sigma)(\transpose\nabla u^{1,T, n,\varepsilon}(s,x)\sigma - \transpose\nabla u^{2,T, n,\varepsilon}(s,x)\sigma) }{|\transpose\nabla u^{1,T, n,\varepsilon}(s,x)\sigma - \transpose\nabla u^{2,T, n,\varepsilon}(s,x)\sigma|^2}\mathds{1}_{s < T},\\
 & \text{if } \nabla u^{1,T, n,\varepsilon}(s,x) \neq \nabla u^{2,T, n,\varepsilon}(s,x),\\
0, & \text{otherwise}.
\end{array}
\right.
\end{align*}
The process $(\beta^T(s,X_s^{x,n,\varepsilon}))_{s \in [0,T]}$ is progressively measurable and bounded, therefore, we can apply Girsanov's Theorem to obtain that there exists a new probability measure $\Q^T$ equivalent to $\Pb$ under which $(W_t-\int_0^t\beta(s,X_s^{x,n,\varepsilon}) \der s)_{t \in [0,T]}$ is a Brownian motion. Therefore, denoting by $E^{\Q^T}$ the expectation with respect to the probability $\Q^T$,
\begin{align*}
u^{1,T, n,\varepsilon}(0,x) -  u^{2,T, n,\varepsilon}(0,x) &= \E^{\Q^T}\left[(v^1-v^2)(X_T^{x,n,\varepsilon}) \right]\\
& = \mathscr{P}_T[v^1-v^2](x),
\end{align*}
where $\mathscr{P}_t$ is the Kolmogorov semigroup associated to the following SDE, $\forall t \geq 0$,
\begin{align*}
 U_t^x = x + \int_0^t\left( \widetilde{b}_\varepsilon + (F_n)_\varepsilon \right)(U_s^{x}) \der s + \int_0^t \sigma \beta(s,U_s^x) \der s + \int_0^s \sigma \der W_s.
\end{align*}
By Corollary \ref{basic coupling estimates étendu} and Remark \ref{approximation upsilon}, we deduce that
\begin{align*}
|u^{1,T, n,\varepsilon}(0,x) -  u^{2,T, n,\varepsilon}(0,x) - (u^{1,T, n,\varepsilon}(0,0) -(u^{2,T, n,\varepsilon}(0,0))| \leq Ce^{-\hat{\eta}T}.
\end{align*}
Therefore, thanks to \eqref{convergence de Y1 pour l unicite EDSRE},
\begin{align*}
|v^1(x) - v^2(x) -(v^1(0) - v^2(0))| \leq Ce^{- \hat{\eta} T}.
\end{align*}
Therefore, since $v^1(0) = v^2(0) = 0$, letting $T \rightarrow + \infty$ we deduce that
\begin{align*}
v^1(x) = v^2(x), \forall x \in \G.
\end{align*}
\end{proof}

We recall the link of such EBSDE with ergodic PDE. Let us consider the following ergodic semilinear PDE for which the unknown is a pair $(v,\lambda)$:
\begin{align}\label{PDE_ergodic}
\left\{
\begin{array}{ll}
\LL v(x) + f(x, \nabla v(x) \sigma) - \lambda = 0, &\forall x \in G,\\
\frac{\partial v(x)}{\partial n} +g(x) =0, &\forall x \in \partial G.
\end{array}
\right.
\end{align}

\begin{lemm}[Existence of ergodic viscosity solutions]
Assume that the Hypotheses \ref{hypo SDE reflected}, \ref{hypo_G} and \ref{hypo EBSDE} hold true then the solution $(v , \lambda)$ of Lemma \ref{lemm_existence_EBSDE} is a viscosity solution of $(\ref{PDE_ergodic})$.
\end{lemm}
\begin{proof}
Note that $v$ is continuous by Lemma \ref{lemm_existence_EBSDE}.  The proof of this result is very classical and can be easily adapted from \cite{PARDOUX_ZHANG_GENERALIZED}.
\end{proof}

\begin{lemm}[Uniqueness of ergodic viscosity solutions]\label{Uniqueness of ergodic viscosity solutions}
Assume that the Hypotheses \ref{hypo SDE reflected}, \ref{hypo_G}, \ref{hypo unicite parabolic PDE} and \ref{EBSDE} hold true. Then uniqueness holds for viscosity solutions $(v,\lambda)$ of \eqref{PDE_ergodic} in the class of (continuous)  viscosity solutions such that $\exists a \in \R^d$, $v^1(a)=v^2(a)$. 
\end{lemm}
\begin{proof}
Let $(v^1,\lambda^1)$ and $(v^2,\lambda^2)$ be two continuous viscosity solutions of \eqref{PDE_ergodic}. First we show that $\lambda^1 = \lambda^2$.
Let us fix $0 \leq  t < T < + \infty$, and let us consider $(Y^{1,T,t,x},Z^{1,T,t,x})$ the solution of the following BSDE in finite horizon, $\forall s \in [t,T]$,
\begin{align*}
Y_s^{1,T,t,x} = v^{1}(X_T^{t,x}) + \int_s^T [f(X_r^{t,x},Z_r^{1,T,t,x}) - \lambda^1] \der r +\int_s^T g(X_r^{t,x}) \der K_r^{t,x}  -  \int_s^T Z_r^{1,T,t,x} \der W_r.
\end{align*}
And we define $(Y^{2,T,t,x},Z^{2,T,t,x})$ similarly, replacing $\lambda^1$ by $\lambda^2$ and $v^1$ by $v^2$.
By Lemma \ref{lemme exitence sol de viscosité parabolique}, we deduce that $u^{1,T}(t,x) = Y^{1,T,t,x}_t$ is a viscosity solution of \eqref{PDE finite horizon}. Since $v^1$ is also a viscosity solution of \eqref{PDE finite horizon} with $h=v^1$, it follows from Lemma \ref{lemme unicite sol de viscosité parabolique} that $\forall t \in [0,T], \forall x \in \G$, 
\begin{align*}
u^{1,T}(t,x) = v^{1}(x).
\end{align*}
Of course, similarly, $\forall t \in [0,T]$, $\forall x \in \G$,
\begin{align*}
u^{2,T}(t,x) = v^2(x).
\end{align*}
Then, taking $t=0$, $\forall T > 0$,
\begin{align*}
u^{1,T}(0,x) - u^{2,T}(0,x) &= v^{1}(X_T^x) - v^2(X_T^x) + \int_0^T [f(X_s^x,Z_s^{1,T,x}) - f(X_s^x,Z_s^{2,T,x})]  \der s\\ &~~~~~~+(\lambda^2-\lambda^1)T - \int_0^T \left(Z_s^{1,T,x} - Z_s^{2,T,x}\right) \der W_s\\
&= v^{1}(X_T^x) - v^2(X_T^x) - \int_0^T (Z_s^{1,T,x} - Z_s^{2,T,x}) ( -\beta_s \der s + \der W_s),
\end{align*}
where, $\forall s \in [0,T]$,
\begin{align*}
\beta_s = 
\left\{
\begin{array}{ll}
\frac{(f(X_s^x,Z_s^{1,T,x}) - f(X_s^x,Z_s^{2,T,x})) \transpose (Z_s^{1,T,x} - Z_s^{2,T,x})}{|Z_s^{1,T,x} - Z_s^{2,T,x}|^2}, & \text{ if } Z_s^{1,T,x} \neq Z_s^{2,T,x},\\
0, & \text{otherwise}.
\end{array}
\right.
\end{align*}
Since $(\beta_s)_{s \in [0,T]}$ is a progressively measurable and bounded process, by Girsanov's theorem, there exists a new probability $\Q^T$ equivalent to $\Pb$ under which $( W_t - \int_0^t \beta_s \der s)_{t \in [0,T]}$ is a Brownian motion. Taking the expectation with respect to this new probability, we get
\begin{align*}
\frac{u^{1,T}(0,x) - u^{2,T}(0,x)}{T} = \frac{\E^{Q^T}(v^{1}(X_T^x) - v^{2}(X_T^x))}{T} + \lambda^2 - \lambda^1.
\end{align*}
Since $v^1$ and $v^2$ are continuous and therefore bounded on $\G$, letting $T \rightarrow + \infty$ we deduce that
\begin{align*}
\lambda^1 = \lambda^2.
\end{align*}
Applying the same argument as that in Theorem \ref{Uniqueness EBSDE}, we deduce the uniqueness.
\end{proof}

\section{Large time behaviour}
\subsection{First behaviour}
We recall that $(Y_s^{T,x},Z_s^{T,x})_{s \geq 0}$ denotes the solution of the finite horizon BSDE (\ref{BSDE}) with $t=0$ and that $(Y_s^{x},Z_s^{x},\lambda)_{s \geq 0}$ denotes the solution of the EBSDE (\ref{EBSDE}).

\begin{théo}\label{theoreme first behaviour}
Assume that the Hypotheses \ref{hypo SDE reflected}, \ref{hypo_G} and \ref{hypo_path_dependent} hold true (path dependent case), then, $\forall x \in \G$, $\forall T > 0$:
\begin{align*}
\left| \frac{Y_0^{T,x}}{T} - \lambda \right| \leq \frac{C}{T}.
\end{align*}
In particular,
\begin{align*}
\frac{Y_0^{T,x}}{T} \underset{T \rightarrow + \infty}{\longrightarrow} \lambda,
\end{align*}
uniformly in $\G$.\\
Assume that the Hypotheses \ref{hypo SDE reflected}, \ref{hypo_G}, \ref{hypo BSDE markovien} and \ref{hypo unicite parabolic PDE} hold true (Markovian case). Then, $\forall x \in \G$, $\forall T > 0$:
\begin{align*}
\left| \frac{Y_0^{T,x}}{T} - \lambda \right| \leq \frac{C}{T}.
\end{align*}
i.e.
\begin{align*}
\left| \frac{u(T,x)}{T} - \lambda \right|\leq \frac{C}{T},
\end{align*}
where $u$ is the viscosity solution of \eqref{HJB Neumann cauchy changement de temps}. In particular,
\begin{align*}
\frac{u(T,x)}{T} = \frac{Y_0^{T,x}}{T} \underset{T \rightarrow + \infty}{\longrightarrow} \lambda,
\end{align*}
uniformly in $\G$.
\end{théo}
\begin{proof}
The proof is identical to the proof of Theorem 4.1 in \cite{LARGE_TIME_BEHAVIOUR_HU_MADEC_RICHOU}. Note that the proof is even simpler since we work with a bounded subset $G$ of $\R^d$ and then for any probability $\Q^T$,  $\E^{\Q^T}[\sup_{0 \leq t \leq T} |X_t|^\mu] \leq C$, where $C$ depends only on $G$ and $\mu$. Note that the proof gives an important result
\begin{align}\label{estimee 1 sur w_T}
|u_T(0,x) - \lambda T - v(x)| \leq C,
\end{align}
which will be useful for what follow.
Finally note that for the Markovian case, Hypothesis \ref{hypo unicite parabolic PDE} is added in order to obtain uniqueness of viscosity solutions of \eqref{HJB Neumann cauchy changement de temps}.
\end{proof}

\subsection{Second and third behaviour}
In this section we introduce a new set of hypothesis without loss of generality. Note that it is the same as Hypothesis $\ref{hypo BSDE markovien}$ but with modified assumptions for $b$. However we write it again for reader's convenience. The remark immediately following this new set of hypothesis justifies the fact that there is no loss of generality.
Let us denote by $(Y_s^{t,x},Z_s^{t,x},\lambda)_{s \geq 0}$ the solution of the EBSDE (\ref{EBSDE}) when $X^x$ is replaced by $X^{t,x}$. We recall that this solution satisfies
\begin{align}\label{Zs = Zs1 + Zs2}
Z_s^{t,x} = \nabla v^1(X_s^{t,x}) \sigma + Z_s^2.
\end{align}

\begin{hypo}\label{hypo_second_troisieme_comportement}
There exists $C > 0$ such that
\begin{enumerate}
\item $b : \R^d \rightarrow \R^d$ is $\mathscr{C}^1$ Lipschitz and dissipative (i.e. $\exists \eta >0$ such that $\forall x,y \in \R^d$, $\langle b(x) - b(y), x-y \rangle \leq -\eta|x-y|^2$).
\item $\xi^T = h(X_T^x)$, where $h : \G \rightarrow \R $ is continuous.
\item $\forall x \in \G$, $\forall z,z' \in \R^{1 \times k}$, $|f(x,z)-f(x,z')|\leq C|z-z'|$.
\item $\forall z \in \R^{1 \times k}$, $f(\cdot,z)$ is continuous.
\item $g \in \mathscr{C}^1_{\text{lip}}(\G)$.
\end{enumerate}
\end{hypo}
\begin{rema}\label{rema changement b en dissipatif}
Note that assuming $b$ to be $\mathscr{C}^1$ Lipschitz and dissipative is not restrictive. Indeed, let us consider $b : \G \rightarrow \R^d$ only Lipschitz. Let us recall that the purpose of this paper is to study the large time behaviour of the viscosity solution of
\begin{align*}
\left\{ 
\begin{array}{ll}
\frac{\partial u(t,x)}{\partial t} = \LL u(t,x) + f(x,\nabla u(t,x)G) , & \forall (t,x) \in \R_+ \times \G,\\
\frac{\partial u(t,x)}{\partial n} + g(x) = 0, & \forall (t,x) \in \R_+ \times \partial G,\\
u(0,x) = h(x),&\forall x \in G.
\end{array}
\right.
\end{align*}
Now, we define, $\forall x \in \R^d$, $\widetilde{b}(x) := -x + (b(\Pi(x)) +\Pi(x))$. Note that $\widetilde{b}$ is equal to $b$ on $\G$. Furthermore,
\begin{align*}
\langle \widetilde{b}(x) , \nabla u(t,x) \rangle + f(x,\nabla u(t,x) \sigma) = \langle -x , \nabla u(t,x) \rangle + \widetilde{f}(x, \nabla u(t,x) \sigma),
\end{align*}
where $\widetilde{f}(x,z) = f(x,z) + \langle b(\Pi(x)) +\Pi(x) , z \sigma^{-1} \rangle$ is a continuous function in $x$ and Lipschitz in $z$.
Therefore, under our assumptions, we can always consider the case $b$ being $\mathscr{C}^1$ Lipschitz and dissipative by replacing $b$ by $(x \mapsto -x)$ and $f$ by $\widetilde{f}$ if necessary.
\end{rema}

\begin{théo}\label{theoreme second and third behaviour}
Assume that the Hypotheses \ref{hypo SDE reflected}, \ref{hypo_G}, \ref{hypo unicite parabolic PDE} and \ref{hypo_second_troisieme_comportement} hold true. Then there exists $L \in \R$ such that,
\begin{align*}
\forall x \in \G, ~Y_0^{T,x} - \lambda T - Y_0^x \underset{T \rightarrow + \infty}{\longrightarrow} L,
\end{align*}
i.e.
\begin{align*}
\forall x \in \G,~ u(T,x) - \lambda T - v(x) \underset{T \rightarrow + \infty}{\longrightarrow} L,
\end{align*}
where $u$ is the viscosity solution of \eqref{HJB Neumann cauchy changement de temps} and $v$ is the viscosity solution of \eqref{PDE_ergodic}.
Furthermore the following rate of convergence holds
\begin{align*}
|Y_0^{T,x}-\lambda T - Y_0^x - L| \leq Ce^{-\hat{\eta} T},
\end{align*}
i.e. 
\begin{align*}
|u_T(0,x)-\lambda T - v(x) - L| \leq Ce^{-\hat{\eta} T}.
\end{align*}
\end{théo}
\begin{proof}
Let us start by defining
\begin{align*}
&u_T(t,x) := Y_t^{T,t,x}\\
&w_T(t,x) := u_T(t,x) - \lambda (T-t) - v(x).
\end{align*}
We recall that $Y_s^{T,t,x} = u_T(s,X_s^{t,x})$ and that $Y_s^x = v(X_s^x)$.

Note that $(x \mapsto w_T(0,x))$ is continuous and bounded uniformly in $T$ by (\ref{estimee 1 sur w_T}). Therefore one can extend the definition of $w_T(0,x)$ to the whole $\R^d$ into a continuous and uniformly bounded in $T$ function by setting $w_{T,\R^d}(0,x) := w_T(0, \Pi(x))$ where $\Pi$ is the projection on $\G$.

Let us first state the following proposition whose proof is relegated to the next subsection.

\begin{prop}\label{prop1:report2}
\begin{align}\label{changement de temps w}
w_T(0,x) = w_{T+S}(S,x), \forall x \in \G.
\end{align} 
\end{prop}

For every $T \geq t$, the process $(w_T(s, X_s^{t,x}))_{s \in [t,T]}$  satisfies the following BSDE in infinite horizon, $\forall t \leq s \leq T < + \infty$,
\begin{align}
w_T(s,X_s^{t,x}) &= w_T(T,X_T^{t,x}) + \int_s^{T} [f(X_r^{t,x},Z_r^{T,t,x}) - f(X_r^{t,x},Z_r^{t,x})] \der r \nonumber \\
& ~~~~~~~~~~~~~~~~~~~~ - \int_s^T (Z_r^{T,t,x} - Z_r^{t,x}) \der W_r \nonumber \\
&= h(X_T^{t,x}) - v(X_T^{t,x})  + \int_s^{T} [f(X_r^{t,x},Z_r^{T,t,x}) - f(X_r^{t,x},Z_r^{t,x})] \der r \nonumber \\
& ~~~~~~~~~~~~~~~~~~~~~~~~~~~  - \int_s^T (Z_r^{T,t,x} - Z_r^{t,x}) \der W_r.\label{equation verifiee par w}
\end{align}

Since we do not have a basic coupling estimate  for the reflected process $X^{t,x}$, we will use an approximation procedure. We fix infinitely differentiable functions $\rho_{\varepsilon} : \R^d \rightarrow \R^+$ bounded together with their derivatives of all order, such that: $\int_{\R^d}\rho_{\varepsilon}(x) \der x = 1$ and
\begin{align*}
\text{supp}(\rho_{\varepsilon}) \subset \left\{ \xi \in \R^d: |\xi| \leq \varepsilon \right\}
\end{align*}
where $\text{supp}$ denotes the support. Then we define $\forall n \in \N$,
\begin{align*}
(F_n)_{\varepsilon} (x) = \int_{\R^d} \rho_{\varepsilon}(y) F_n(x-y) \der y.
\end{align*}
It is well known that $(F_n)_\varepsilon$ is $\mathscr{C}^{\infty}$. Furthermore, $(F_n)_{\varepsilon}$ is still $0$-dissipative. 
Let $(X_s^{t,x, n, \varepsilon})_{ s \geq t}$ be the solution of
\begin{align*}
X_s^{t,x,n,\varepsilon} = x + \int_t^s (b + (F_n)_{\varepsilon})(X_r^{t,x,n,\varepsilon}) \der r  + \int_t^s \sigma \der W_r,~~~\forall s \geq t,
\end{align*}
and $(Y_s^{2,t,x,\alpha,n,\varepsilon}, Z_s^{2,t,x, \alpha,n,\varepsilon})_{s \geq t}$ be the solution of the following monotonic BSDE in infinite horizon, $\forall t \leq s \leq T < +\infty$,
\begin{align*}
Y_s^{2,t,x,\alpha,n,\varepsilon } = Y_T^{2,t,x,\alpha,n,\varepsilon} &+ \int_s^T \left[f(X_r^{t,x,n,\varepsilon},Z_r^{2,t,x,\alpha,n,\varepsilon }) -  \alpha Y_r^{2,t,x,\alpha,n,\varepsilon}\right] \der r\\
&  - \int_s^T Z_r^{2,t,x,\alpha,n,\varepsilon } \der W_r.
\end{align*}
By the same argument as that of Theorem 4.4 in \cite{MADEC_ERGODIC}, there exist sequences $\varepsilon_m \underset{m \rightarrow + \infty}{ \longrightarrow} 0$, $\beta(n) \underset{n \rightarrow + \infty}{\longrightarrow} + \infty$ and $\alpha_k \underset{k \rightarrow + \infty}{\longrightarrow} 0$ such that for all $T \geq t$, 
\begin{align}\label{convergence dans M2 du Z ergodique}
\lim_{k \rightarrow + \infty} \lim_{n \rightarrow + \infty} \lim_{m \rightarrow + \infty} \E \int_t^T \left|Z_s^{2,t,x,\alpha_k, \beta(n), \varepsilon_m } -  Z_s^{2}\right|^2 \der s = 0.
\end{align}
In what follows, we will use the following notation. If $q^{\alpha , n ,\varepsilon}$ denotes a function depending on the parameters $\alpha$, $n$ and $\varepsilon$, then
\begin{align*}
\lim_{\alpha, n, \varepsilon} q^{\alpha , n, \varepsilon} := \lim_{k \rightarrow + \infty} \lim_{n \rightarrow + \infty} \lim_{m \rightarrow + \infty} q^{\alpha_k , \beta(n), \varepsilon_m } .
\end{align*}

Now, if we define, for all $s \geq t$,
\begin{align*}
\widetilde{Z}_s^{t,x,\alpha,n,\varepsilon} := (\nabla v^1)_{\R^d}(X_s^{t,x,n,\varepsilon}) \sigma  + Z_s^{2,t,x, \alpha, n,\varepsilon},
\end{align*}
by the dominated convergence theorem and thanks to \eqref{Zs = Zs1 + Zs2} and \eqref{convergence dans M2 du Z ergodique},  for all $T \geq t$
\begin{align}\label{convergence Ztilde vers Z}
\lim_{\alpha, n,\varepsilon}\E\int_t^T |\widetilde{Z}_s^{t,x,\alpha,n,\varepsilon}  - Z_s^{t,x}|^2 \der s = 0.
\end{align}

Note that by Theorem 4.2 in \cite{MA_ZHANG_REPRESENTATION_THEOREMS_FOR_BACKWARD_STOCHASTIC_DIFFERENTIAL_EQUATIONS}, if we define $v^{2, \alpha , n, \varepsilon}(x) := Y_0^{x, \alpha , n,\varepsilon}$, then $v^{2, \alpha, n ,\varepsilon}$ is $\mathscr{C}^1$ and $\forall s \geq t$,
\begin{align*}
Z_s^{2,t,x, \alpha , n,\varepsilon} = \transpose\nabla v^{2, \alpha , n, \varepsilon }(X_s^{t,x, n,\varepsilon}) \sigma.
\end{align*}
Therefore, we have the following representation, $\forall s \geq t$,
\begin{align}\label{representation Zs alpha n epsilon}
\widetilde{Z}_s^{t,x, \alpha, n,\varepsilon } &= \transpose\nabla (v^1)_{\R^d}(X_s^{t,x, n,\varepsilon}) \sigma +  \transpose\nabla v^{2, \alpha , n,\varepsilon}(X_s^{t,x, n,\varepsilon}) \sigma \nonumber \\ 
& =: \transpose\nabla \widetilde{v}^{\alpha , n,\varepsilon } (X_s^{t,x, n,\varepsilon}) \sigma.
\end{align}
Let us denote by $(\overline{Y}_s^{T,t,x,\alpha,n,\varepsilon}, \overline{Z}_s^{T,t,x,\alpha,n,\varepsilon})_{s \geq t}$ the solution of the following BSDE in finite horizon, $\forall s \in [t,T]$,
\begin{align}
\overline{Y}_s^{T,t,x,\alpha,n,\varepsilon} &= w_{T, \R^d}(T,X_{T}^{t,x,n,\varepsilon})  - \int_s^{T} \overline{Z}_r^{T,t,x,\alpha,n,\varepsilon} \der W_r \nonumber \\
&~~~~  + \int_s^{T} \left[ f_{\R^d}(X_r^{t,x,n,\varepsilon} , \overline{Z}_r^{T,t,x,\alpha,n,\varepsilon} + \widetilde{Z}_r^{t,x,\alpha,n,\varepsilon}) - f_{\R^d}(X_r^{t,x,n,\varepsilon},\widetilde{Z}_r^{t,x,\alpha,n,\varepsilon})  \right] \der r \nonumber \\
&= (h-v)_{\R^d}(X_{T}^{t,x,n,\varepsilon}) + \int_s^{T} \widetilde{f}^{\alpha , n,\varepsilon}(s, \overline{Z}_r^{T,t,x,\alpha,n,\varepsilon} ) \der r - \int_s^{T} \overline{Z}_r^{T,t,x,\alpha,n,\varepsilon} \der W_r, \label{equation verifiee par Y alpha, n , epsilon}
\end{align}
where, for all $r \geq t$,  $z \in \R^{1\times d}$,
\begin{align*}
\widetilde{f}^{\alpha,n,\varepsilon}(r,z) :=  f_{\R^d}(X_r^{t,x,n,\varepsilon} , z + \widetilde{Z}_r^{t,x,\alpha,n,\varepsilon}) - f_{\R^d}(X_r^{t,x,n,\varepsilon},\widetilde{Z}_r^{t,x,\alpha,n,\varepsilon}).
\end{align*}

We define, for all $z \in \R^{1 \times d}$,
\begin{align*}
\widetilde{f}(r, z) := f_{\R^d}(X_r^{t,x} , z + Z_r^{t,x})  - f_{\R^d}(X_r^{t,x},Z_r^{t,x}).
\end{align*}

Now we can apply the stability theorem in \cite{BRIAND_HU_STABILITY_BSDE} to show the following convergence result:
\begin{prop}\label{prop2:report2}$\forall x \in \G$,
\begin{align}\label{convergence vers wT}
\lim_{\alpha, n,\varepsilon } \overline{Y}_t^{T,t,x,\alpha, n,\varepsilon} = w_{T}(t,x).
\end{align}
\end{prop}

The proof of this proposition will be given in the next subsection.
\medskip
We define 
\begin{align*}
\overline{w}_T^{\alpha, n , \varepsilon}(t,x) = \overline{Y}_t^{T,t,x,\alpha, n,\varepsilon }.
\end{align*}
Similarly to equation \eqref{changement de temps w}, we deduce that, $\forall T,S \geq 0$
\begin{align}\label{eq changement de temps w alpha n epsilon}
\overline{w}_{T}^{\alpha, n , \varepsilon}(0,x) = \overline{w}_{T+S}^{\alpha, n , \varepsilon}(S,x).
\end{align}

Now we are in force to apply the method exposed in \cite{LARGE_TIME_BEHAVIOUR_HU_MADEC_RICHOU} for the quantity $\overline{w}_T^{\alpha, n , \varepsilon}(0,x)$ with slight modifications.

First we establish the following proposition whose proof is also relegated to the next subsection.
\begin{prop}\label{lemm estimates Y alpha n epsilon}
Under the hypotheses of Theorem \ref{theoreme second and third behaviour}, $\exists C >0$, $\forall x,y \in \G$, $\forall T > 0$, $\forall 0 < T' \leq T$, $\exists C_{T'}>0$,
\begin{align*}
&|\overline{w}_T^{\alpha, n , \varepsilon}(0,x)| \leq C,\\
&|\nabla_x \overline{w}_T^{\alpha, n , \varepsilon}(0,x)| \leq \frac{C_{T'}}{\sqrt{T'}}, \\
&| \overline{w}_T^{\alpha, n , \varepsilon}(0,x) -  \overline{w}_T^{\alpha, n , \varepsilon}(0,y)| \leq Ce^{- \hat{\eta} T}.
\end{align*} 
We stress the fact that $C$ depends only on $\eta$, $\sigma$, $\G$. The constant $C_{T'}$ depends only on the same constants and $T'$.
\end{prop}

Let us conclude the proof. From Proposition \ref{lemm estimates Y alpha n epsilon}, we derive, by the same arguments as in \cite{LARGE_TIME_BEHAVIOUR_HU_MADEC_RICHOU} that there exists $L^{\alpha,n,\varepsilon} \in \R$ such that $\forall x \in \R^d$,
\begin{align}\label{resultat avec alpha n varepsilon}
|\overline{w}_T^{\alpha,n,\varepsilon}(0,x) - L^{\alpha,n,\varepsilon}| \leq Ce^{-\hat{\eta}T}.
\end{align}
Therefore,
\begin{align*}
\left\{
\begin{array}{l}
\overline{w}_T^{\alpha,n,\varepsilon}(0,x) \leq Ce^{- \hat{\eta}T} + L^{\alpha,n,\varepsilon},\\
L^{\alpha,n,\varepsilon} \leq Ce^{- \hat{\eta}T} + \overline{w}_T^{\alpha,n,\varepsilon}(0,x),
\end{array}
\right.
\end{align*}
which implies by \eqref{convergence vers wT} that
\begin{align*}
\left\{
\begin{array}{l}
w_T(0,x) \leq Ce^{- \hat{\eta}T} + \liminf_{\alpha,n,\varepsilon}L^{\alpha,n,\varepsilon},\\
\limsup_{\alpha,n,\varepsilon}L^{\alpha,n,\varepsilon} \leq Ce^{- \hat{\eta}T} + w_T(0,x).
\end{array}
\right.
\end{align*}
Then, 
\begin{align*}
\limsup_{\alpha,n,\varepsilon}L^{\alpha,n,\varepsilon} \leq 2Ce^{-\hat{\eta}T} + \liminf_{\alpha,n,\varepsilon} L^{\alpha,n,\varepsilon}.
\end{align*}
Letting $T \rightarrow + \infty$ implies that there exists $L \in \R$ such that
\begin{align*}
\lim_{\alpha, n , \varepsilon} L^{\alpha,n,\varepsilon} = L.
\end{align*}
Coming back to \eqref{resultat avec alpha n varepsilon} and passing to the limit gives us the result:
\begin{align*}
|w_T(0,x) -L| \leq  Ce^{- \hat{\eta}T}.
\end{align*}
\end{proof}

\begin{rema}\label{add1:report1}
As in \cite{LARGE_TIME_BEHAVIOUR_HU_MADEC_RICHOU}, in this paper we need to suppose that $\sigma$ is a constant matrix in order to establish the a priori estimates in Lemma \ref{lemm estimates Y alpha n epsilon}.
\end{rema}

\subsection{Proofs of the propositions}

{\bf Proof of Proposition \ref{prop1:report2}. }
We recall that for all $T,S \geq 0$, $u_T$ is the unique solution of 
\begin{align*}
\left\{ 
\begin{array}{ll}
\frac{\partial u_T(t,x)}{\partial t} + \LL u_T(t,x) + f(x,\nabla u_T(t,x)\sigma) = 0, & \forall (t,x) \in [0,T] \times G,\\
\frac{\partial u_T(t,x)}{\partial n} + g(x) =0,& \forall (t,x) \in [0,T]\times \partial G,\\
u_T(T,x) = h(x),&\forall x \in G,
\end{array}
\right.
\end{align*}
and that 
 $u_{T+S}$ is the unique solution of 
\begin{align*}
\left\{ 
\begin{array}{ll}
\frac{\partial u_{T+S}(t,x)}{\partial t} + \LL u_{T+S}(t,x) + f(x,\nabla u_{T+S}(t,x)\sigma) = 0, & \forall (t,x) \in [0,T+S] \times G,\\
\frac{\partial u_{T+S}(t,x)}{\partial n} + g(x) =0,& \forall (t,x) \in [0,T+S]\times \partial G,\\
u_{T+S}(T+S,x) = h(x),&\forall x \in G.
\end{array}
\right.
\end{align*}
By uniqueness of viscosity solutions, it implies that $u_T(0,x)=u_{T+S}(S,x)$, for all $x \in \G$, and then,
$$w_T(0,x) = w_{T+S}(S,x), \forall x \in \G.$$

{\bf Proof of Proposition \ref{prop2:report2}.} It suffices to check that the Assumptions (A2) and (A3)
in Lemma 2.3 of \cite{BRIAND_HU_STABILITY_BSDE}
are verified for  the equations \eqref{equation verifiee par Y alpha, n , epsilon} and \eqref{equation verifiee par w}. Let us first give precisely (A2) and (A3).

Assumption (A2) of \cite{BRIAND_HU_STABILITY_BSDE} is:
\begin{align*}
&\forall z_1, z_2 \in \R^{1\times d}, |\widetilde{f}^{\alpha,n,\varepsilon}(s,z_1) - \widetilde{f}^{\alpha,n,\varepsilon}(s,z_2)| \leq C|z_1 - z_2|,\\
&\E \left[ \int_t^T |\widetilde{f}^{\alpha, n,\varepsilon}(s,0)|^2 \der s \right] \leq C;
\end{align*}
and Assumption (A3) is:
\begin{align*}
\lim_{\alpha, n, \varepsilon }\E&\left[ \left(\int_s^T (\widetilde{f}^{\alpha ,n,\varepsilon}(r,Z_r^{T,t,x}-Z_r^{t,x}) - \widetilde{f}(r,Z_r^{T,t,x}-Z_r^{t,x})) \der r\right)^2 \right]=0,
\end{align*}
\begin{align*}
\lim_{\alpha, n, \varepsilon } \E \left[ |(h-v)_{\R^d}(X_{T}^{t,x,n,\varepsilon}) - (h-v)_{ \R^d}(0,X_{T}^{t,x}) |^2 \right] = 0.
\end{align*}

It is easy to show that (A2) is satisfied. It remains to show 
that  (A3)  is also satisfied. We have, $\forall  s \in [t,T]$,
\begin{align*}
\E&\left[ \int_s^T |\widetilde{f}^{\alpha ,n,\varepsilon}(r,Z_r^{T,t,x}-Z_r^{t,x}) - \widetilde{f}(r,Z_r^{T,t,x}-Z_r^{t,x})|^2 \der r \right] \\
&= \E \left[\int_s^T |f_{\R^d}(X_r^{t,x,n,\varepsilon} , Z_r^{T,t,x} - Z_r^{t,x} + \widetilde{Z}_r^{t,x,\alpha,n,\varepsilon}) - f_{\R^d}(X_r^{t,x,n,\varepsilon},\widetilde{Z}_r^{t,x,\alpha,n,\varepsilon})\right. \\
&~~~~~~ \left. - f_{\R^d}(X_r^{t,x} , Z_r^{T,t,x})  + f_{\R^d}(X_r^{t,x},Z_r^{t,x})|^2 \der r \right ]\\
&\leq C\E\left[ \int_s^T |f_{\R^d}(X_r^{t,x,n,\varepsilon} , Z_r^{T,t,x} - Z_r^{t,x} + \widetilde{Z}_r^{t,x,\alpha,n,\varepsilon}) - f_{\R^d}(X_r^{t,x,n,\varepsilon} , Z_r^{T,t,x}) |^2 \der r \right]\\
&~~~~+ C\E\left[ \int_s^T |f_{\R^d}(X_r^{t,x,n,\varepsilon} , Z_r^{T,t,x}) - f_{\R^d}(X_r^{t,x} , Z_r^{T,t,x})|^2  \der r \right]\\
&~~~~ + C\E\left[ \int_s^T | f_{\R^d}(X_r^{t,x,n,\varepsilon},\widetilde{Z}_r^{t,x,\alpha,n,\varepsilon}) - f_{\R^d}(X_r^{t,x,n,\varepsilon} , Z_r^{t,x}) |^2 \der r \right] \\
&~~~~+ C\E\left[ \int_s^T |f_{\R^d}(X_r^{t,x,n,\varepsilon} , Z_r^{t,x}) - f_{\R^d}(X_r^{t,x} , Z_r^{t,x})|^2 \der r \right]\\
& \leq C\E \left[ \int_s^T |\widetilde{Z}_r^{t,x,\alpha,n,\varepsilon} - Z_r^{t,x}|^2 \der r \right] + C\E \left[ \int_s^T|f_{\R^d}(X_r^{t,x,n,\varepsilon} , Z_r^{T,t,x}) - f_{\R^d}(X_r^{t,x} , Z_r^{T,t,x})|^2  \der r \right] \\
& ~~~~+ C \E \left[ \int_s^T |f_{\R^d}(X_r^{t,x,n,\varepsilon} , Z_r^{t,x}) - f_{\R^d}(X_r^{t,x} , Z_r^{t,x})|^2 \der r \right].\\
\end{align*}
Then \eqref{convergence Ztilde vers Z} implies that the first term converges toward $0$. Furthermore, since\\ $\lim_{\varepsilon , n} \E \sup_{t \leq s \leq T} |X_s^{t,x,n,\varepsilon} - X_s^{t,x}|^2  = 0$, we have 
\begin{align*}
|X_s^{t,x,n,\varepsilon} - X_s^{t,x}| \overset{\Pb \otimes \der t}{\longrightarrow} 0, \text{ as } \varepsilon \rightarrow 0, n \rightarrow + \infty,
\end{align*}
and
\begin{align*}
|f_{\R^d}(X_r^{t,x,n,\varepsilon} , Z_r^{T,t,x}) - f_{\R^d}(X_r^{t,x} , Z_r^{T,t,x})|^2 \leq C(1+|Z_r^{T,t,x}|^2)
\end{align*}
which shows the uniform integrability of $|f_{\R^d}(X_r^{t,x,n,\varepsilon} , Z_r^{T,t,x}) - f_{\R^d}(X_r^{t,x} , Z_r^{T,t,x})|^2$. Therefore, the second term converges toward $0$. The same argument applied to the third term shows that this last term also converges toward $0$.

Furthermore, by continuity and boundedness of $(h-v)_{\R^d}$ , we deduce that:
\begin{align*}
\lim_{\alpha, n, \varepsilon } \E \left[ |(h-v)_{\R^d}(X_{T}^{t,x,n,\varepsilon}) - (h-v)_{ \R^d}(0,X_{T}^{t,x}) |^2 \right] = 0.
\end{align*}
Thus (A3) is satisfied. Therefore, by Lemma 2.3 of \cite{BRIAND_HU_STABILITY_BSDE} applied to \eqref{equation verifiee par Y alpha, n , epsilon} and \eqref{equation verifiee par w}, we obtain:
\begin{align*}
\lim_{\alpha, n ,\varepsilon} \overline{Y}_t^{T,t,x,\alpha, n,\varepsilon } = w_{T, \R^d}(t,x).
\end{align*}
Thus, $\forall x \in \G$,
\begin{align*}
\lim_{\alpha, n,\varepsilon } \overline{Y}_t^{T,t,x,\alpha, n,\varepsilon} = w_{T}(t,x).
\end{align*}

{\bf Proof of Proposition \ref{lemm estimates Y alpha n epsilon}.}
The first estimate is a direct consequence of Girsanov's theorem. Indeed, we have,
\begin{align*}
\overline{Y}_0^{T,x,\alpha,n,\varepsilon} &= (h-v)_{\R^d}(X_{T}^{x,n,\varepsilon}) - \int_0^{T} \overline{Z}_r^{T,x,\alpha,n,\varepsilon} \der W_r  \\
&~~~~+ \int_0^{T} \left[ f_{\R^d}(X_r^{x,n,\varepsilon} , \overline{Z}_r^{T,x,\alpha,n,\varepsilon} + \widetilde{Z}_r^{x,\alpha,n,\varepsilon}) - f_{\R^d}(X_r^{x,n,\varepsilon},\widetilde{Z}_r^{x,\alpha,n,\varepsilon})  \right] \der r\\
&=w_{T}(0,X_{T}^{x,n,\varepsilon}) - \int_0^{T} \overline{Z}_r^{T,x,\alpha,n,\varepsilon} (-\beta_r \der r + \der W_r),
\end{align*}
where 
\begin{align*}
\beta_r := 
\left\{
\begin{array}{ll}
\frac{ (f_{\R^d}(X_r^{x,n,\varepsilon} , \overline{Z}_r^{T,x,\alpha,n,\varepsilon} + \widetilde{Z}_r^{x,\alpha,n,\varepsilon}) - f_{\R^d}(X_r^{x,n,\varepsilon},\widetilde{Z}_r^{x,\alpha,n,\varepsilon}))\transpose( \overline{Z}_r^{T,x,\alpha,n,\varepsilon})}{| \overline{Z}_r^{T,x,\alpha,n,\varepsilon}|^2}, & \text{if }  \overline{Z}_r^{T,x,\alpha,n,\varepsilon} \neq 0,\\
0, & \text{otherwise.}
\end{array}
\right.
\end{align*}
Since $\beta$ is a progressively measurable and bounded process, there exists a new probability equivalent to $\Pb$, $\Q^{T, \alpha , n ,\varepsilon}$ under which $(W_s - \int_0^{s} \beta_r \der r)_{r \in [0,T]} $ is a Brownian motion. Therefore, thanks to estimate (\ref{estimee 1 sur w_T}):
\begin{align*}
|\overline{Y}_0^{T,x,\alpha, n,\varepsilon}| &\leq \E^{Q^{T, \alpha , n,\varepsilon}}| w_{T}(0 , X_{T}^{x, n,\varepsilon}) |\\
& \leq C.
\end{align*}

Let us establish the second and third inequality of the proposition. First we notice that thanks to equation (\ref{changement de temps w}), $\forall 0 \leq T' < T$, $\forall s \in [t,T']$, 
\begin{align*}
\overline{Y}_s^{T,t,x,\alpha,n,\varepsilon} &= \overline{w}^{\alpha,n,\varepsilon}_{T, \R^d}(T',X_{T'}^{t,x,n,\varepsilon}) - \int_s^{T'} \overline{Z}_r^{T,t,x,\alpha,n,\varepsilon} \der W_r \\
&~~~+ \int_s^{T'} \left[ f_{\R^d}(X_r^{t,x,n,\varepsilon} , \overline{Z}_r^{T,t,x,\alpha,n,\varepsilon} + \widetilde{Z}_r^{t,x,\alpha,n,\varepsilon}) - f_{\R^d}(X_r^{t,x,n,\varepsilon},\widetilde{Z}_r^{t,x,\alpha,n,\varepsilon})  \right] \der r \\
&= \overline{w}^{\alpha,n,\varepsilon}_{T-T' , \R^d}(0,X_{T'}^{t,x,n,\varepsilon}) + \int_s^{T'} \widetilde{f}^{\alpha , n,\varepsilon }(s, \overline{Z}_r^{T,t,x,\alpha,n,\varepsilon} ) \der r - \int_s^{T'} \overline{Z}_r^{T,t,x,\alpha,n,\varepsilon} \der W_r.
\end{align*}

We recall that we have the following representation:
 \begin{align*}
\widetilde{Z}_s^{t,x, \alpha, n,\varepsilon } = \nabla \widetilde{v} (X_s^{t,x, n,\varepsilon}) \sigma.
\end{align*}
Furthermore, by Theorem 4.2 (or Theorem 4.2 in \cite{FUHRMAN_TESSITORE_BISMUT_ELWORTHY}), as $\overline{w}_T^{\alpha , n,\varepsilon}(t,x) := \overline{Y}_{t}^{T,t,x,\alpha ,n,\varepsilon}$, $(x \mapsto \overline{u}_T^{\alpha , n,\varepsilon}(t,x))$ is continuously differentiable for all $t \in [0,T[$ and  $\forall s \in [t,T[$,
\begin{align*}
\overline{Z}_s^{t,x, \alpha , n,\varepsilon} = \nabla \overline{w}_T^{\alpha , n,\varepsilon }(s,X_s^{t, x , n,\varepsilon }) \sigma.
\end{align*} 
 
Therefore, we can apply the same method as exposed in \cite{LARGE_TIME_BEHAVIOUR_HU_MADEC_RICHOU} to obtain the second and third estimate.

\section{Application to an ergodic control problem}
In this section, we show how we can apply our results to an ergodic control problem. We assume that Hypotheses \ref{hypo SDE reflected} and \ref{hypo_G} hold. Let $U$ be a separable metric space. We define a control $a$ as an $(\F_t)_{t \geq 0}$-predictable $U$-valued process. We will assume the following.
\begin{hypo}\label{hypo optimization control}
The functions $R : U \rightarrow \G$, $L : \G \times U \rightarrow \R$ and $h_0 : \G \rightarrow \R$ are measurable and satisfy, for some $C > 0$,
\begin{enumerate}
\item $|R(a)| \leq C,~~~\forall a \in U$.
\item $L(\cdot,a)$ is continuous in $x$ uniformly with respect to $a \in U$.  Furthermore $|L(x,a)| \leq C,~~~\forall x \in \G,  \forall a \in U$.
\item $h_0(\cdot)$ is continuous.
\item $g \in \mathscr{C}^1(\G)$.
\end{enumerate}
\end{hypo}
We denote by $(X_t^x)_{t \geq 0}$ the solution of \eqref{SDE_reflected}. Given an arbitrary control $a$ and $T>0$, we introduce the Girsanov density
\begin{align*}
\rho_T^{x,a} = \exp\left( \int_0^T \sigma^{-1}R(a_s) \der W_s - \frac{1}{2}\int_0^T |\sigma^{-1} R(a_s)|^2 \der s \right)
\end{align*}
and the probability $\Pb_T^a = \rho_T^a \Pb$ on $\F_T$. We introduce two costs. The first one is the cost in finite horizon:
\begin{align*}
J^T(x,a) := \E^{a,T}\left[\int_0^T L(X_s^x,a_s) \der s + \int_0^T g(X_s^x) \der K_s^x \right] + \E^{a,T} h_0(X_T^x),
\end{align*}
where $\E^{a,T}$ denotes the expectation with respect to $\Pb_T^{a}$. The associated optimal control problem is to minimize the cost $J^T(x,a)$ over all controls $a^T : \Omega \times [0,T] \rightarrow U$, progressively measurable. The second one is called the ergodic cost and is the time averaged finite horizon cost:
\begin{align*}
J(x,a) := \limsup_{T \rightarrow + \infty} \frac{1}{T} \E^{a,T} \left[ \int_0^T L(X_s^x,a_s) \der s + \int_0^T g(X_s^x) \der K_s^x \right].
\end{align*}
The associated optimal control problem is to minimize the cost $J(x,a)$ over all controls $a : \Omega \times [0,+ \infty[ \rightarrow + \infty$, progressively measurable.

We notice that $W_t^a = W_t - \int_0^t \sigma^{-1} R(a_s) \der s$ is a Brownian motion on $[0,T]$ under $\Pb_T^a$ and that
\begin{align*}
\der X_t^x = (b(X_t^x)+R(a_t))\der t + \sigma \der W_t^a+\nabla\phi(X_t^x)dK_t^x, ~~~ \forall t \in [0,T], 
\end{align*}
and this justifies our formulation of the control problem.

We want to show how our results can be applied to such an optimization problem to get an asymptotic expansion of the finite horizon cost involving the ergodic cost.

To apply our results, we first define the Hamiltonian in the usual way,
\begin{align}\label{hamiltonien}
f_0(x,z) = \inf_{a \in U}\left\{ L(x,a) + z \sigma^{1} R(a) \right\},
\end{align}
and we note that , if for all $x,z$ the infimum is attained in \eqref{hamiltonien}, then by the Filippov theorem (see \cite{MCSHANE_FILIPPOV_IMPLICIT_FUNCTION_LEMMA}), there exists a measurable function $\gamma : \G \times \R^{1 \times d}$ such that
\begin{align*}
f_0(x,z) = L(x,\gamma(x,z)) + z\sigma^{-1} R(\gamma(x,z)).
\end{align*}

\begin{lemm}
Under the above assumptions, the Hamiltonian $f_0$ satisfies assumptions on $f$ in Hypotheses \ref{hypo_path_dependent}, \ref{hypo BSDE markovien}, \ref{hypo EBSDE}, or \ref{hypo_second_troisieme_comportement}.
\end{lemm}
\begin{proof}
See Lemma 5.2 in \cite{FUHRMAN_TESSITORE_BISMUT_ELWORTHY}.
\end{proof}
We recall the following results about the finite horizon cost:
\begin{lemm}
Assume that Hypotheses \ref{hypo SDE}, \ref{hypo_G}, \ref{hypo unicite parabolic PDE} and \ref{hypo optimization control} hold true. Then for arbitrary control $a^T : \Omega \times [0,T] \rightarrow U$,
\begin{align*}
J^T (x,a^T) \geq u(T,x),
\end{align*}
where $u(t,x)$ is the viscosity solution of 
\begin{align*}
\left\{ 
\begin{array}{ll}
\frac{\partial u(t,x)}{\partial t} = \LL u(t,x) + f_0(x,\nabla u(t,x)G) , & \forall (t,x) \in \R_+ \times G,\\
\frac{\partial u(t,x)}{\partial n} + g(x) = 0, &\forall (t,x) \in \R_+ \times \partial G, \\
u(0,x) = h_0(x),&\forall x \in G,
\end{array}
\right.
\end{align*}
Furthermore, if $\forall x,$ $z$ the infimum is attained in \eqref{hamiltonien} then we have the equality:
\begin{align*}
J^T(x, \overline{a}^T) = u(T,x),
\end{align*}
where $\overline{a}^T_t = \gamma(X_t^{x},\nabla u(t,X_t^x) \sigma)$.
\end{lemm}
\begin{proof}
The proof of this result is similar to the proof of Theorem 7.1 in \cite{FUHRMAN_HU_TESSITORE_ERGODIC_BSDE}, so we omit it.
\end{proof}
Similarly, for the ergodic cost we have the following result.
\begin{lemm}
Assume that Hypotheses \ref{hypo SDE reflected}, \ref{hypo_G}, \ref{hypo unicite parabolic PDE} and \ref{hypo optimization control} hold true, then for arbitrary control $a : \Omega \times [0,+\infty[ \rightarrow U$,
\begin{align*}
J(x,a) \geq \lambda,
\end{align*}
where $(v,\lambda)$ is the viscosity solution of 
\begin{align*}
\left\{
\begin{array}{l}
\LL v + f_0(x,\nabla v(x) \sigma) -\lambda = 0, ~~~~\forall x \in G,\\
\frac{\partial v(t,x)}{\partial n} + g(x) = 0, ~~~~\forall x \in \partial G .
\end{array}
\right.
\end{align*}
Furthermore, if $\forall x,$ $z$ the infimum is attained in \eqref{hamiltonien} then we have the equality:
\begin{align*}
J^T(x,\overline{a}) = \lambda,
\end{align*}
where $\overline{a}_t = \gamma(X_t^x,\nabla v(X_t^x) \sigma).$
\end{lemm}

Finally, we apply our result to obtain the following theorem.
\begin{théo}
Assume that Hypotheses \ref{hypo SDE reflected}, \ref{hypo_G}, \ref{hypo unicite parabolic PDE} and \ref{hypo optimization control} hold true. Then, for any control $a : \Omega \times [0,T] \rightarrow U$, we have
\begin{align*}
\liminf_{T \rightarrow + \infty} \frac{J^T(x,a^T)}{T} \geq \lambda.
\end{align*}
Furthermore, if $\forall x,$ $z$ the infimum is attained in \eqref{hamiltonien} then
\begin{align*}
|J^T(x,\overline{a}^T) - J(x,\overline{a}) T - v(x) + L| \leq Ce^{-\hat{\eta} T}.
\end{align*}
\end{théo}
\begin{proof}
The proof is a straightforward consequence of the two previous lemmas above and of Theorem \ref{theoreme second and third behaviour}.
\end{proof}

{\bf Acknowledgement} The authors thank the two referees and the editor for their careful reading
and helpful suggestions which lead to a much improved version of this paper.


\newpage

\end{document}